\documentclass[a4paper,10pt]{article}
\usepackage{amssymb}
\usepackage{filecontents}
\usepackage{latexsym}
\usepackage{amsmath}
\usepackage{amsthm}
\usepackage{amsfonts}
\usepackage{amssymb}
\usepackage{amsmath,amscd}
\usepackage{graphicx}
\usepackage{xcolor}
\usepackage{hyperref}

\newtheorem{Th}{Theorem}

\newtheorem{Co}{Corollary}
\newtheorem{Lm}{Lemma}
\newtheorem{Lma}{Lemma}[section]
\newtheorem{Dfi}{Definition}
\newtheorem{Rm}{Remark}
\newtheorem{Que}{Question}
\newcommand{\be}{\begin{equation}}
\newcommand{\ee}{\end{equation}}
\newcommand{\bes}{\begin{equation*}}
\newcommand{\ees}{\end{equation*}}

\newcommand{\R}{\mathbb{R}}
\newcommand{\N}{\mathbb{N}}
\newcommand{\C}{\mathbb{C}}
\newcommand{\Z}{\mathbb{Z}}

\newcommand\res{\mathop{\hbox{\vrule height 7pt width .5pt depth 0pt
\vrule height .5pt width 6pt depth 0pt}}\nolimits}

\newcommand{\reset}{\setcounter{equation}{0}\setcounter{Th}{0}\setcounter{Prop}{0}\setcounter{Co}{0}
\setcounter{Lm}{0}\setcounter{Rm}{0}}

\def\Xint#1{\mathchoice
{\XXint\displaystyle\textstyle{#1}}%
{\XXint\textstyle\scriptstyle{#1}}%
{\XXint\scriptstyle\scriptscriptstyle{#1}}%
{\XXint\scriptscriptstyle\scriptscriptstyle{#1}}%
\!\int}
\def\XXint#1#2#3{{\setbox0=\hbox{$#1{#2#3}{\int}$}
\vcenter{\hbox{$#2#3$}}\kern-.5\wd0}}
\def\dashint{\Xint-}

\def\La{\Lambda}
\def\La{\Lambda}
\def\ti{\tilde}
\def\lf{\left}
\def\rg{\right}

\def\la{\lambda}

\def\ds{\displaystyle}
\def\ov{\overline}
\def\Om{\Omega}
\def\om{\omega}
\def\p{\partial}

\def\res{\mathop{\hbox{\vrule height 7pt width .5pt 
depth 0pt\vrule height .5pt width 6pt depth 0pt}}\nolimits}

\makeatletter
\newcommand{\tpitchfork}{%
  \vbox{
    \baselineskip\z@skip
    \lineskip-.52ex
    \lineskiplimit\maxdimen
    \m@th
    \ialign{##\crcr\hidewidth\smash{$-$}\hidewidth\crcr$\pitchfork$\crcr}
  }%
}
\makeatother

\begin{document}

\title{A Variational Construction of Hamiltonian Stationary Surfaces with Isolated Schoen--Wolfson Conical Singularities}

\author{Filippo Gaia, Gerard Orriols and Tristan Rivi\`ere\footnote{Department of Mathematics, ETH Zentrum,
CH-8093 Z\"urich, Switzerland.}}

\maketitle
{\bf Abstract :}{\it We construct using variational methods Hamiltonian Stationary Surfaces with Isolated Schoen--Wolfson Conical Singularities. We obtain these surfaces through a convergence process reminiscent to the Ginzburg--Landau asymptotic analysis in the strongly repulsive regime introduced by Bethuel, Brezis and H\'elein. We describe in particular how the prescription of Schoen--Wolfson conical singularities is related to optimal Wente constants.}

\noindent{\bf Math. Class. 53D12, 49Q05, 58E12, 49Q10, 35J50, 35J25, 35J65, 53C42}

\section{Introduction}
\label{sec:introduction}
The main object of study in the present work are \emph{Lagrangian stationary surfaces}. These surfaces are critical points of the area  among surfaces satisfying the Lagrangian constraint
\be
\label{0.1}
\iota_\Sigma^\ast\Om=0\ \quad\mbox{ where }\quad \Om= dx_1\wedge dx_2+dx_3\wedge dx_4,
\ee
and where $\iota_\Sigma$ is the canonical embedding of the surface $\Sigma$ into ${\C}^2$.

In a breakthrough paper \cite{SW} Richard Schoen and Jon Wolfson discovered that Lagrangian oriented surfaces $\Sigma$ in ${\C}^2$ which are {\it Lagrangian stationary} can have isolated point singularities at which the Gauss map\footnote{The Gauss map is the map which assigns at every point the oriented tangent 2-plane. This is a map from $\Sigma$ into the 3 dimensional Lagrangian Grassmanian $\La(2)\simeq U(2)/ O(2)$, sub-manifold  of the Grassmann manifold  $G_2({\R}^4)\simeq {\C}P^1\times {\C}P^1$ } is not continuous. This stands in contrast with classical stationary surfaces for the area (or minimal surfaces) in ${\C}^2$ which can degenerate only at  branched points at which the Gauss map is still continuous. These isolated point singularities discovered by Schoen and Wolfson are asymptotically given by cones parametrized by maps of the form 
\be
\label{0.2}
\Phi_{p,q}:D^2\to \mathbb{C}^2,\quad (r,\theta)\mapsto \frac{r^{\sqrt{pq}}}{\sqrt{p+q}} \begin{pmatrix}\sqrt{q}e^{ip\theta}
\\ i\sqrt{p}e^{-iq\theta}\end{pmatrix}.
\ee
called nowadays Schoen--Wolfson cones. A careful analysis of this expression shows that the pullback by the Gauss map, taking values into the Lagrangian Grassmann manifold $\Lambda(2)$, of the generator $H^1(\Lambda(2), {\Z})$ realizes a non trivial element in $H^1(D^2\setminus\{0\},{\Z})$ called \emph{Maslov index}\footnote{The Maslov Index of the cone~\ref{0.2} is $p-q$ (see the preliminaries).} (see for instance \cite{BaLe}). Some of these cones,  the ones for which $p-q=\pm 1$, are shown in \cite{SW} to be stable with respect to area variations under the pointwise Lagrangian constraint (\ref{0.1}). Later, J.~Wolfson in \cite{Wol}  constructed minimizers in some spherical Lagrangian homology class of a Ricci flat complex surface which admits minimizers with ``co-existing'' isolated Schoen-Wolfson conical singularities of degree $\pm1$. The question which has motivated the present work is, a Lagrangian  stationary surface in ${\C}^2$ being given, 
\[
\mbox{\it is there any restriction on the location of the Schoen--Wolfson conical singularities?}
\]
We are not able at this stage to give a satisfying answer to this question but we will turn it differently by constructing surfaces (mostly discs) with isolated Schoen--Wolfson conical singularities for which the location will be strongly correlated to the boundary data.

\medskip

Precisely a map $\Phi$ is a conformal Lagrangian immersion of the disc $D^2$ in ${\C}^2$ if and only if $\Phi$ is conformal and the multiplication by $i$ realizes an isometry from the tangent plane to the normal plane of the immersion. A direct computation (see the preliminaries) gives that this implies  the existence of an $S^1$ valued map $g$ such that
\be
\label{0.3}
\Phi^\ast (dz\wedge dw)= {g}^{-1}\ dvol_{\Phi}\quad\mbox{in }\ D^2\ .
\ee
A structural equation for the immersion to be conformal Lagrangian is
\be
\label{0.4}
\mbox{ div}(g\,\nabla \Phi)=0\quad\mbox{in }\ D^2\ .
\ee
while the condition for being a critical point of the area (for variations up to the boundary) is
\be
\label{0.5}
\lf\{
\begin{array}{l}
\ds \operatorname{div}\lf[ g^{-1}\,\nabla g\rg]=0\quad\mbox{in }D^2\\[5mm]
\ds g^{-1}\,\p_{r}g=0\quad\mbox{ in }\p D^2\ .
\end{array}
\rg.
\ee
The coupling of the structural equation (\ref{0.4}) and the Euler-Lagrange Equation (\ref{0.5}) is called {\it Hamiltonian stationary Equation}. Area variations under Lagrangian constraints
were first consider by Yong-Geun Oh in \cite{Oh1}. 

Our main result is the following.
\begin{Th}
\label{th-0.1}
Let $(d_l)_{l=1\cdots N}$ such that $d_l=\pm1$. Let $p_l$ be $N$ distinct points in $D^2$ and let $G$ be the Green's function solution to 
\be
\label{0.6}
\lf\{
\begin{array}{rl}
\ds\Delta G&\ds=2\pi\ \sum_{l=1}^N\ d_l\, \delta_{p_l}\quad\mbox{in }{\mathcal D}'(D^2)\\[5mm]
\ds G&\ds=0\quad\mbox{ on }\quad\p D^2\ .
\end{array}
\rg.
\ee
Assume each of the connected components of $G\ne 0$ is a disc containing exactly one $p_l$. Let $\psi\in C^1(\p D^2,{\C})$. Then there exists a solution $$(\Phi,g)\in \bigcap_{p<2}W^{1,p}(D^2,{\C}^2)\times W^{1,p}(D^2, S^1)$$ smooth away from the $p_l$ of the {\it Hamiltonian Stationary Equation}
\be
\label{0.7}
\lf\{
\begin{array}{l}
\mbox{ div}(g\,\nabla \Phi)=0\quad\mbox{in }\ D^2\ \\[5mm]
\Phi=\psi \quad\mbox{ in }\p D^2\\[5mm] 
\ds \operatorname{div}\lf[ g^{-1}\,\nabla g\rg]=0\quad\mbox{in }D^2\\[5mm]
\ds g^{-1}\,\p_{r}g=0\quad\mbox{ in }\p D^2\ .
\end{array}
\rg.
\ee
with 
\[
g^{-1}\nabla g=i\,\nabla^\perp G\ .
\] 
where $\nabla^\perp G:=(-\p_{x_2}G,\p_{x_1}G)$. Moreover $\Phi$ is conformal on $D^2$ and on each connected component $\om$ of $G\ne 0$ and there exists $\ti{\Phi}\in W^{1,2}(\om,{\C}^2)$ and $A^\om\in {\C}$ such that
\be
\label{0.8}
\Phi=\ti{\Phi}+A^\om\, g\ .
\ee
This solution $(\Phi,g)$ is unique in this class moreover $\ti{\Phi}\in C^{1,\sqrt{2}-1}(D^2)$. There exists a sub-space of co-dimension at most $N$ of $\psi\in H^{1/2}(\p D^2,{\C})$ such that $\Phi$ has finite area and realize an Hamiltonian Stationary Conformal Immersion with isolated Schoen--Wolfson conical Singularities of Maslov degree $d_l$ at the $p_l$.
\hfill $\Box$
\end{Th}
\begin{Rm}
\label{rm-0.2} In the cases $N=1,2$ we are able to prove that the subspace of $\psi\in H^{1/2}(\p D^2,{\C})$ such that $\Phi$ has finite area has exactly co-dimension 2.\hfill $\Box$
\end{Rm}

\medskip

Without Lagrangian constraint, the most natural and direct way to produce conformally parametrized area stationary map from a disc to ${\C}^2$ goes as follows. One first solve
\be
\label{0.9}
\lf\{
\begin{array}{l}
\ds\Delta u=0\quad\mbox{ in }D^2\\[5mm]
\ds u=\psi\quad\mbox{ on }\p D^2
\end{array}
\rg.
\ee
(where $u$ and $v$ take values in $\mathbb{C}$) then one introduce the harmonic conjugate $v$ of $u$ satisfying $\nabla^\perp v=\nabla u$ and we observe that
\[
\Phi:=\lf(\begin{array}{c} u\\[3mm] \ov{v}\end{array}\rg)
\]
is harmonic conformal and hence realizes a conformal minimal immersion of the disc $D^2$. For proving the theorem~\ref{th-0.1} we extend this classical approach to the hamiltonian framework that is, we first solve
\be
\label{0.10}
\lf\{
\begin{array}{l}
\ds\operatorname{div}\lf( g\, \nabla u\rg)=0\quad\mbox{ in }D^2\\[5mm]
\ds u=\psi\quad\mbox{ on }\p D^2
\end{array}
\rg.
\ee
then we introduce the harmonic conjugate $v$ of $u$ satisfying $\nabla^\perp v=g\,\nabla u$ and we realizes that
\[
\Phi:=\lf(\begin{array}{c} u\\[3mm] \ov{v}\end{array}\rg)
\]
is a solution to our problem\footnote{Observe that following the approach in lemma III.4 of \cite{DR}, introducing the corresponding quaternion
\[
{\mathbf \Phi}= u+\ov{v}\,{\mathbf j}=u+{\mathbf j}\,{v}
\]
 where the complex number is identified withe the imaginary quaternion ${\mathbf i}$, we obtain that $\mathbf \Phi$ solves the {\it Quaternionic Cauchy Riemann} type equation
 \[
 \p_x {\mathbf \Phi}-g\,{\mathbf j}\,\p_{y} {\mathbf \Phi}=0\ .
 \]}. All the difficulty amounts in solving (\ref{0.10}). This equation happens to be variational and are the critical points of 
\be
\label{0.11}
 L(u):=\frac{1}{2}\int_{D^2}|\nabla u|^2\ dx^2+\frac{1}{2}\int_{D^2}G\, \lf<i\,\nabla u,\nabla^\perp u\rg>\ dx^2\ .
\ee
A natural approach would consist in minimizing $L$ but, under the assumptions on $G$ we are making, the functional $L$ is convex but not coercive thanks to the 
optimal constant in Wente inequality (see lemma~\ref{lma-1} in the Appendix). We introduce instead the perturbed functional
\be
\label{0.12}
L_t(u):=\frac{1}{2}\int_{D^2}|\nabla u|^2\ dx^2+\frac{t}{2}\int_{D^2}G\, \lf<i\,\nabla u,\nabla^\perp u\rg>\ dx^2\ .
\ee
for $t<1$. This functional is also convex and it is easy to obtain a unique minimizer $u_t$ of 
\be
\label{0.13}
E_t(\psi):=\inf\lf\{
\begin{array}{c}
\ds L_t(u):=\frac{1}{2}\int_{D^2}|\nabla u|^2\ dx^2+\frac{t}{2}\int_{D^2}G\, \lf<i\,\nabla u,\nabla^\perp u\rg>\ dx^2\\[5mm]
\ds u\in W^{1,2}(D^2,{\C})\quad\mbox{ and }\quad u=\psi\quad\mbox{ on }\p D^2
 \end{array} \rg\}.
\ee
The idea consists then in making $t\rightarrow 1$ but as in the Ginzburg-Landau type asymptotic analysis of F.~Bethuel, H.~Brezis and F.~H\'elein \cite{BBH} we are facing the problem of a lack of a-priori estimate since there exists $\Theta_g(\psi)$ such that
\be
\label{0.14}
\int_{D^2}|\nabla u_t|^2\ dx^2=\frac{\Theta_g(\psi)}{\sqrt{1-t}}+O(1)\,\longrightarrow \,+\infty\quad\mbox{ if }\Theta_g(\psi)\ne 0
\ee
In the case of one point, $N=1$, we have an explicit expression of $\Theta_g(\psi)$:
\be
\label{0.15}
\Theta_g(\psi)=\pi\,\lf|\int_{\p D^2}\psi\, dg\rg|^2.
\ee
In a similar way as in \cite{BBH} and subsequent versions of the proof of their main theorem, in particular the one in \cite{ST}, we prove that the blow-up of the $L^2$-norm of $\nabla u_t$ is in fact ``hiding'' a control of $\nabla u_t$ in the slightly bigger Marcinkiewicz space $L^{2,\infty}(D^2)$.

\medskip

\subsection{Open Problems}
\begin{itemize}
\item Can Theorem \ref{th-0.1} be generalised to construct Hamiltonian stationary Lagrangian discs with prescribed Schoen--Wolfson singularities in Calabi-Yau surfaces?
\item It would be interesting to investigate the same question for closed surfaces: can one construct closed Hamiltonian stationary surfaces (in ${\C}^2$, in Calabi-Yau manifolds or even in K\"ahler--Einstein manifolds) with isolated Schoen--Wolfson singularities? Closed Hamiltonian stationary surfaces have been constructed, among others, by I.~Castro and F.~Urbano \cite{CU}, F.~H\'elein and P.~Romon \cite{HR-cag, HR-cmh}, H.~Anciaux \cite{An}, but in all cases they construct branched immersions without cone singularities.
\item Does there exist complete Hamiltonian Stationary planes in ${\C}^2$ with  isolated singularities beside the Schoen--Wolfson cones (\ref{0.2})? If yes, are they stable 
with respect to Hamiltonian deformations?
\item Is the sufficient condition on $G$ given by theorem~\ref{th-0.1} necessary for the existence of finite area Hamiltonian Stationary discs in ${\C}^2$ with $\pm 1$ Schoen--Wolfson cones? Is there some notion of renormalized energy in the style of the Ginzburg--Landau renormalized energy? Can one compute  explicitly the function $\Theta_g$ for $N>2$?
\item Is there a counterpart to theorem~\ref{th-0.1} for higher multiplicities ($|d_l|>1$)?
\item Is it possible to construct Hamiltonian stationary Lagrangian discs with locally finite area but infinitely many Schoen--Wolson singularities accumulating at the boundary?
\end{itemize}

\section{Preliminaries}
\reset
\label{sec:preliminaries}
\subsection{The Hamiltonian Stationary Equation}

Let $\Phi=(\Phi_1,\Phi_2)$ be a conformal Lagrangian immersion from $D^2$ into ${\C}^2$, that is we have
\be
\label{p-1}
\lf\{
\begin{array}{l}
|\p_{x_1}\Phi|^2=|\p_{x_2}\Phi|^2=e^{2\la}\\[5mm]
\p_{x_1}\Phi\cdot\p_{x_2}\Phi=0
\end{array}
\rg.
\ee
and
\be
\label{p-2}
0=2\,\Phi^\ast\Om=-\,d\Phi\,\dot{\wedge}\,id\Phi=d\Phi_1\wedge i\,d\ov{\Phi}_1+d\Phi_2\wedge i\,d\ov{\Phi}_2
\ee
where
\be
\label{p-3}
\Om=dx_1\wedge dx_2+dx_3\wedge dx_4= \frac{1}{2} \lf[dz_1\wedge i\,d\ov{z}_1+dz_2\wedge i\,d\ov{z}_2\rg]
\ee
\subsubsection{The Lagrangian Angle Function}
We have
\be
\label{p-4}
\begin{array}{l}
\ds0=2\ \Phi^\ast\Om=-\,d\Phi\,\dot{\wedge}\,id\Phi\\[5mm]
\ds=\lf<\p_{x_1} \Phi_1\,,\,\p_{x_2}i \Phi_1\rg>-\lf<\p_{x_2} \Phi_1\,,\,\p_{x_1}i \Phi_1\rg>+\lf<\p_{x_1} \Phi_2\,,\,\p_{x_2}i \Phi_2\rg>-\lf<\p_{x_2} \Phi_2\,,\,\p_{x_1}i \Phi_2\rg>\ dx_1\wedge dx_2\\[5mm]
\ds=-\,i\ \p_{x_1} \Phi_1\,\p_{x_2}\ov{\Phi}_1\ +\,i\,\p_{x_1} \ov{\Phi}_1\,\p_{x_2}{\Phi_1}-\,i\ \p_{x_1} \Phi_2\,\p_{x_2}\ov{\Phi}_2\ +\,i\,\p_{x_1} \ov{\Phi}_2\,\p_{x_2}{\Phi_2}\ dx_1\wedge dx_2
\end{array}
\ee
Hence
\be
\label{V.15}
\p_{x_1} \Phi_1\,\p_{x_2}\ov{\Phi_1}\ -\,\p_{x_1} \ov{\Phi_1}\,\p_{x_2}{\Phi_1}+\ \p_{x_1} \Phi_2\,\p_{x_2}\ov{\Phi_2}\ -\p_{x_1} \ov{\Phi_2}\,\p_{x_2}{\Phi_2}=0
\ee
Observe that, since $\Phi$ is conformal we have also
\be
\label{V.16}
\p_{x_1}\Phi_1\,\p_{x_2}\ov{\Phi}_1+\p_{x_1}\ov{\Phi}_1\,\p_{x_2}{\Phi}_1+\p_{x_1}\Phi_2\,\p_{x_2}\ov{\Phi}_2+\p_{x_1}\ov{\Phi}_2\,\p_{x_2}{\Phi}_2=0
\ee
Combining (\ref{V.15}) and (\ref{V.16}) gives then
\be
\label{V.17}
\p_{x_1}\Phi_1\,\p_{x_2}\ov{\Phi}_1+\p_{x_1}\Phi_2\,\p_{x_2}\ov{\Phi}_2=0\ .
\ee
We compute (without using the Lagrangian condition (\ref{V.17}))
\be
\label{V.18}
\begin{array}{l}
\ds\lf|\p_{x_1}\Phi_1\,\p_{x_2}\Phi_2-\p_{x_2}\Phi_1\,\p_{x_1}\Phi_2\rg|^2\\[5mm]
\ds=\lf|\p_{x_1}\Phi_1\rg|^2\,\lf|\p_{x_2}\Phi_2\rg|^2+\lf|\p_{x_1}\Phi_2\rg|^2\,\lf|\p_{x_2}\Phi_1\rg|^2-2\, \lf<\p_{x_1}\Phi_1\,\p_{x_2}\Phi_2,\p_{x_2}\Phi_1\,\p_{x_1}\Phi_2\rg>
\end{array}
\ee
We write now using (\ref{V.17})
\be
\label{V.19}
\begin{array}{l}
\ds-\,2\, \lf<\p_{x_1}\Phi_1\,\p_{x_2}\Phi_2,\p_{x_2}\Phi_1\,\p_{x_1}\Phi_2\rg>=-\p_{x_1}\Phi_1\,\p_{x_2}\ov{\Phi}_1\,\p_{x_2}\Phi_2\,\p_{x_1}\ov{\Phi}_2-\p_{x_2}\Phi_1\,\p_{x_1}\ov{\Phi}_1\,\p_{x_1}\Phi_2\,\p_{x_2}\ov{\Phi}_2\\[5mm]
\ds=|\p_{x_1} \Phi_2|^2\,|\p_{x_2}\Phi_2|^2+|\p_{x_1} \Phi_1|^2\,|\p_{x_2}\Phi_1|^2=2\, |\p_{x_1} \Phi_2|^2\,|\p_{x_2}\Phi_2|^2
\end{array}
\ee
Combining (\ref{V.18}) and (\ref{V.19}) gives
\be
\label{V.20}
\begin{array}{l}
\ds\lf|\p_{x_1}\Phi_1\,\p_{x_2}\Phi_2-\p_{x_2}\Phi_1\,\p_{x_1}\Phi_2\rg|^2\\[5mm]
\ds=|\p_{x_2}\Phi_2|^2\,\lf[\lf|\p_{x_1}\Phi_1\rg|^2+|\p_{x_1} \Phi_2|^2\rg]+|\p_{x_1}\Phi_2|^2\,\lf[\lf|\p_{x_2}\Phi_1\rg|^2+|\p_{x_2} \Phi_2|^2\rg]=e^{4\,\la}
\end{array}
\ee
Observe
 \be
\label{p-6}
\Phi^\ast(dz_1\wedge dz_2)=\p_{x_1}\Phi_1\,\p_{x_2}\Phi_2-\p_{x_2}\Phi_1\,\p_{x_1}\Phi_2\ dx_1\wedge dx_2\, \ .
\ee
Then we have proved the following lemma
\begin{Lm}
\label{lm-angle}
Let $\Phi$ be a conformal Lagrangian immersion of $D^2$ into ${\C}^2$ then there exists an $S^1$  valued map $g\ :\ D^2\rightarrow S^1$ such that
\be
\label{p-7}
\Phi^\ast(dz_1\wedge dz_2)=g^{-1}\ dvol_{\Phi}\ \quad\mbox{ on }D^2
\ee
where $dvol_\Phi$ is the volume form on $D^2$ induced by the immersion. The map $g^{-1}$ is called Lagrangian angle.\hfill $\Box$
\end{Lm}
\begin{Rm}
\label{rm-angle} Observe that the map
\be
\label{p-7-b}
L\ :\ x\in D^2\ \longrightarrow\  e^{-\la}\,\lf(\begin{array}{cc} \p_{x_1}\ov{\Phi}_1&\p_{x_2}\ov{\Phi}_1\\[3mm] \p_{x_1}\ov{\Phi}_2&\p_{x_2}\ov{\Phi}_2\end{array}\rg)
\ee
defines a map into $U(2)$ whose determinant giving the Maslov index is equal to $g$. \hfill $\Box$
\end{Rm}
\subsubsection{The Lagrangian Angles of the Schoen Wolfson Cones}
We consider the immersion $\Phi_{pq}$ of a Schoen-Wolfson cone of the form (\ref{0.2}). First we observe
\be
\label{p-8}
\lf\{
\begin{array}{l}
\ds |\p_r\Phi_{pq}|^2=\frac{p\,q}{p+q}\, r^{2\sqrt{pq}-2}\,(p+q)=p\,q\, r^{2\sqrt{pq}-2}\\[5mm]
\ds r^{-2}\,|\p_\theta\Phi_{pq}|^2=\frac{r^{2\sqrt{pq}-2}}{p+q}\, (q\,p^2+p\,q^2)=p\,q\, r^{2\sqrt{pq}-2}\\[5mm]
\ds \p_r\Phi_{pq}\cdot \p_\theta\Phi_{pq}=0
\end{array}
\rg.
\ee
Hence the parametrization $\Phi_{pq}$ is conformal. We have
\be
\label{p-9}
\begin{array}{l}
\ds \Phi_{pq}^\ast(dz_1\wedge dz_2)=\frac{i\,\sqrt{p\,q}}{p+q} d\lf(r^{\sqrt{pq}} e^{i\,p\,\theta}\rg)\wedge d\lf(r^{\sqrt{pq}} e^{-\,i\,q\,\theta}\rg)\\[5mm]
\ds=\frac{\sqrt{p\,q}}{p+q}\ q\ \sqrt{pq}\ r^{2\,\sqrt{pq}-2} \ e^{i(p-q)\theta}\, r\, dr\wedge d\theta-\frac{\sqrt{p\,q}}{p+q}\ p\ \sqrt{pq}\   e^{i(p-q)\theta}\,r^{2\,\sqrt{pq}-2} \ r\,d\theta\wedge dr\\[5mm]
\ds=e^{i(p-q)\theta}\, p\,q\, r^{2\sqrt{pq}-2}\, r\, dr\wedge d\theta=e^{i(p-q)\theta}\ dvol_{ \Phi_{pq}}
\end{array}
\ee
Hence the {\it Lagrangian Angle Function} of the Schoen Wolfson cone is
\be
\label{p-10}
g^{-1}:=e^{i(p-q)\theta}\ .
\ee

\subsubsection{The Mean Curvature Vector of a Lagrangian Immersion}
 Let $\Phi=(\Phi_1,\Phi_2)$ be a conformal Lagrangian immersion from $D^2$ into ${\C}^2$ that is $\Phi$ satisfies (\ref{p-1}). The mean curvature vector of $\vec{H}_\Phi$ is given by
 \be
 \label{p-11}
 \vec{H}_\Phi:=\frac{e^{-2\la}}{2}\,\Delta\Phi
 \ee
The mean curvature vector is normal to the tangent space to $\Phi$ generated by $\p_{x_1}\Phi$ and $\p_{x_2}\Phi$. Since $\Phi$ is assumed to be Lagrangian, the multiplication by $i$
realizes a isometry from the tangent space to $\Phi$ and the normal space. Hence there exists $a,b\in {\R}$ such that
\be
\label{p-12}
\frac{e^{-2\la}}{2}\,\Delta\Phi=i\,a\, \p_{x_1}\Phi+i\,b\,\p_{x_2}\Phi\ .
\ee
We aim at computing $a$ and $b$. We have
\be
\label{p-13}
\lf\{
\begin{array}{l}
\ds\p_{x_1}g\, \p_{x_1}\Phi=\p_{x_1}\lf(e^{-2\la}\, \lf(\p_{x_1}\ov{\Phi}_1\,\p_{x_2}\ov{\Phi}_2-\p_{x_2}\ov{\Phi}_1\,\p_{x_1}\ov{\Phi}_2\rg)\, \p_{x_1}\Phi \rg)- g\,\p^2_{x_1^2}\Phi\\[5mm]
\ds\p_{x_2}g\, \p_{x_2}\Phi=\p_{x_2}\lf(e^{-2\la}\, \lf(\p_{x_1}\ov{\Phi}_1\,\p_{x_2}\ov{\Phi}_2-\p_{x_2}\ov{\Phi}_1\,\p_{x_1}\ov{\Phi}_2\rg)\, \p_{x_2}\Phi \rg)- g\,\p^2_{x_2^2}\Phi
\end{array}
\rg.
\ee
This gives for $j=1,2$
\be
\label{p-14}
\lf\{
\begin{array}{l}
\ds\p_{x_1}g\, \p_{x_1}\Phi_1=\p_{x_1}\lf(e^{-2\la}\, \lf(|\p_{x_1}\ov{\Phi}_1|^2\,\p_{x_2}\ov{\Phi}_2-\p_{x_2}\ov{\Phi}_1\,\p_{x_1}\Phi_1\,\p_{x_1}\ov{\Phi}_2\rg)\rg)- g\,\p^2_{x_1^2}\Phi_1\\[5mm]
\ds\p_{x_1}g\, \p_{x_1}\Phi_2=\p_{x_1}\lf(e^{-2\la}\, \lf(\p_{x_1}\ov{\Phi}_1\,\p_{x_2}\ov{\Phi}_2\,\p_{x_1}\Phi_2 -\p_{x_2}\ov{\Phi}_1\,|\p_{x_1}\ov{\Phi}_2|^2\rg)\rg)- g\,\p^2_{x_1^2}\Phi_2\\[5mm]
\ds\p_{x_2}g\, \p_{x_2}\Phi_1=\p_{x_2}\lf(e^{-2\la}\, \lf(\p_{x_1}\ov{\Phi}_1\,\p_{x_2}\Phi_1 \p_{x_2}\ov{\Phi}_2-|\p_{x_2}\ov{\Phi}_1|^2\,\p_{x_1}\ov{\Phi}_2\rg) \rg)- g\,\p^2_{x_2^2}\Phi_1\\[5mm]
\ds\p_{x_2}g\, \p_{x_2}\Phi_2=\p_{x_2}\lf(e^{-2\la}\, \lf(\p_{x_1}\ov{\Phi}_1\,|\p_{x_2}\ov{\Phi}_2|^2-\p_{x_2}\ov{\Phi}_1\,\p_{x_1}\ov{\Phi}_2\p_{x_2}\Phi_2\rg) \rg)- g\,\p^2_{x_2^2}\Phi_2
\end{array}
\rg.
\ee
Using (\ref{V.17}) we obtain
\be
\label{p-15}
\lf\{
\begin{array}{l}
\ds\p_{x_1}g\, \p_{x_1}\Phi_1=\p_{x_1}\lf(e^{-2\la}\, \lf(|\p_{x_1}\ov{\Phi}_1|^2\,\p_{x_2}\ov{\Phi}_2+\p_{x_2}\ov{\Phi}_2\,\p_{x_1}\Phi_2\,\p_{x_1}\ov{\Phi}_2\rg)\rg)- g\,\p^2_{x_1^2}\Phi_1\\[5mm]
\ds\p_{x_1}g\, \p_{x_1}\Phi_2=\p_{x_1}\lf(e^{-2\la}\, \lf(-\,\p_{x_2}\ov{\Phi}_1\ \p_{x_1}\ov{\Phi}_1\ \p_{x_1}{\Phi}_1 -\p_{x_2}\ov{\Phi}_1\,|\p_{x_1}\ov{\Phi}_2|^2\rg)\rg)- g\,\p^2_{x_1^2}\Phi_2
\\[5mm]
\ds\p_{x_2}g\, \p_{x_2}\Phi_1=\p_{x_2}\lf(e^{-2\la}\, \lf(-\p_{x_1}\ov{\Phi}_2\,\p_{x_2}\Phi_2\, \p_{x_2}\ov{\Phi}_2-|\p_{x_2}\ov{\Phi}_1|^2\,\p_{x_1}\ov{\Phi}_2\rg) \rg)- g\,\p^2_{x_2^2}\Phi_1\\[5mm]
\ds\p_{x_2}g\, \p_{x_2}\Phi_2=\p_{x_2}\lf(e^{-2\la}\, \lf(\p_{x_1}\ov{\Phi}_1\,|\p_{x_2}\ov{\Phi}_2|^2+\p_{x_2}\ov{\Phi}_1\,\p_{x_1}\ov{\Phi}_1\p_{x_2}\Phi_1\rg) \rg)- g\,\p^2_{x_2^2}\Phi_2
\end{array}
\rg.
\ee
Hence
\be
\label{p-16}
\lf\{
\begin{array}{l}
\ds\p_{x_1}g\, \p_{x_1}\Phi_1=\p_{x_1}\lf(e^{-2\la}\, |\p_{x_1}\ov{\Phi}|^2\,\p_{x_2}\ov{\Phi}_2\rg)- g\,\p^2_{x_1^2}\Phi_1=\p^2_{x_1 x_2}\ov{\Phi}_2- g\,\p^2_{x_1^2}\Phi_1\\[5mm]
\ds\p_{x_1}g\, \p_{x_1}\Phi_2=-\p_{x_1}\lf(e^{-2\la}\, |\p_{x_1}\ov{\Phi}|^2\,\p_{x_2}\ov{\Phi}_1\rg)- g\,\p^2_{x_j^2}\Phi_2=-\p^2_{x_1 x_2}\ov{\Phi}_1- g\,\p^2_{x_1^2}\Phi_2\\[5mm]
\ds\p_{x_2}g\, \p_{x_2}\Phi_1=-\p_{x_2}\lf(e^{-2\la}\,|\p_{x_2}\Phi|^2\ \p_{x_1}\ov{\Phi}_2\rg)- g\,\p^2_{x_2^2}\Phi_1=-\p^2_{x_1 x_2}\ov{\Phi}_2- g\,\p^2_{x_2^2}\Phi_1\\[5mm]
\ds\p_{x_2}g\, \p_{x_2}\Phi_2=\p_{x_2}\lf(e^{-2\la}\, \p_{x_1}\ov{\Phi}_1\,|\p_{x_2}\ov{\Phi}|^2\rg)- g\,\p^2_{x_2^2}\Phi_2=\p^2_{x_1 x_2}\ov{\Phi}_1- g\,\p^2_{x_2^2}\Phi_2
\end{array}
\rg.
\ee
This implies the following lemma
\begin{Lm}
\label{lm-mean-curvature}
Let $\Phi=(\Phi_1,\Phi_2)$ be a conformal Lagrangian immersion from $D^2$ into ${\C}^2$ and let $g^{-1}$ be its associated Lagrangian angle satisfying
\be
\label{p-17}
\Phi^\ast(dz_1\wedge dz_2)=g^{-1}\ dvol_{\Phi}\ \quad\mbox{ on }D^2
\ee
Then the mean curvature vector of the immersion is given by 
 \be
 \label{p-18}
 \vec{H}_\Phi:=\frac{e^{-2\la}}{2}\,\Delta\Phi=-\frac{e^{-2\la}}{2}\, g^{-1}\,\nabla g\cdot \nabla \Phi=-\frac{e^{-2\la}}{2}\, g^{-1}\,\lf[\p_{x_1}g \,\p_{x_1}\Phi+\p_{x_2}g \,\p_{x_2}\Phi\rg]\ .
 \ee
 \hfill $\Box$
\end{Lm}
This lemma had first been obtained by Dazord~\cite{Daz} and such an expression of the mean curvature of a Lagrangian surface in fact extends to the more general framework of K\"ahler--Einstein manifolds~\cite{Ono}.
\subsubsection{The Lagrangian/Hamiltonian Stationary Equation}
We consider a conformal Lagrangian immersion $\Phi$ from $D^2$ into ${\C}^2$ which is critical point of the area for any infinitesimal variation preserving the Lagrangian condition. Hence we consider $\Phi+t\Psi$ such that
\be
\label{p-19}
\begin{array}{rl}
\lf.\frac{d}{dt}(\Phi+t\Psi)^\ast\Om)\rg|_{t=0}=0\quad&\Longleftrightarrow \quad d\Phi_1\wedge d\Psi_2+d\Psi_1\wedge d\Phi_2+d\Phi_3\wedge d\Psi_4+d\Psi_3\wedge d\Phi_4=0\\[5mm]
 &\Longleftrightarrow \quad d\lf[ \Psi_1\,d\Phi_2-\Psi_2\, d\Phi_1+\Psi_3\,d\Phi_4-\Psi_4\, d\Phi_3\rg]=0\\[5mm]
  &\Longleftrightarrow \quad d\lf[ \Psi \cdot i\,d\Phi\rg]=0
\end{array}
\ee
Hence locally there exists a function $h$, called Hamiltonian function, such that
\be
\label{p-20}
\Psi \cdot i\,d\Phi=dh
\ee
Recall that $e^{-\la}\,(i\,\p_{x_1}\Phi,i\,\p_{x_2}\Phi)$ realizes an orthonormal basis of the normal plane to the immersion. This implies that $\Psi$ has the form
\be
\label{p-21}
\Psi= e^{-2\la}\ \p_{x_1} h\ i\,\p_{x_1}\Phi+e^{-2\la}\ \p_{x_2} h\ i\,\p_{x_2}\Phi+ e^{-2\la}\ a\,\p_{x_1}\Phi+e^{-2\la}\ b\,\p_{x_2}\Phi
\ee
where $a$ and $b$ are real functions. The first variation of the area is given by the mean curvature vector, in particular, using lemma~\ref{lm-mean-curvature}, $\Phi$ is a critical point of the area among Lagrangian map if and only if for any Hamiltonian $h$
\be
\label{p-22}
\begin{array}{l}
\ds 0=\int_{D^2}<\Delta \Phi\cdot \Psi> \ dx^2\\[5mm]
\ds=\int_{D^2}\lf<\lf[g^{-1}\,\p_{x_1}g \,\p_{x_1}\Phi+g^{-1}\,\p_{x_2}g \,\p_{x_2}\Phi\rg]\cdot e^{-2\la}\ \lf[\p_{x_1} h\ i\,\p_{x_1}\Phi+\ \p_{x_2} h\ i\,\p_{x_2}\Phi\rg]\rg>\ dx^2\\[5mm]
\ds= -\,\int_{D^2}\, i\,g^{-1}\,\p_{x_1}g\,\p_{x_1}h +i\,g^{-1}\,\p_{x_2}g\,\p_{x_2}h\ dx^2
\end{array}
\ee
Hence we obtain the following lemma
\begin{Lm}
\label{lm-euler-lagrange}
Let $\Phi=(\Phi_1,\Phi_2)$ be a conformal Lagrangian immersion from $D^2$ into ${\C}^2$ and let $g^{-1}$ be its associated Lagrangian angle satisfying
\be
\label{p-23}
\Phi^\ast(dz_1\wedge dz_2)=g^{-1}\ dvol_{\Phi}\ \quad\mbox{ on }D^2
\ee
Then $\Phi$ is a critical point of the area for any perturbation which preserves infinitesimally the Lagrangian condition if and only if the following equation is satisfied
\be
\label{p-24}
\lf\{
\begin{array}{l}
\ds \operatorname{div}\lf[ g^{-1}\,\nabla g\rg]=0\quad\mbox{in }D^2\\[5mm]
\ds g^{-1}\,\p_{r}g=0\quad\mbox{ in }\p D^2\ .
\end{array}
\rg.
\ee
\hfill $\Box$
\end{Lm}
\begin{Rm}
\label{rm-0.1}
The equation (\ref{p-24}) is the so called $S^1$-harmonic map equation.  We call the system
\be
\label{p-25}
\lf\{
\begin{array}{l}
\ds\mbox{ div}\lf(g\,\nabla\Phi\rg)=0\quad\mbox{in }D^2\\[5mm]
\ds \operatorname{div}\lf[ g^{-1}\,\nabla g\rg]=0\quad\mbox{in }D^2\\[5mm]
\ds g^{-1}\,\p_{r}g=0\quad\mbox{ in }\p D^2
\end{array}
\rg.
\ee
Hamiltonian stationary system, and when the map $\Phi$ is a conformal immersion, solutions to (\ref{p-25}) parametrize Hamiltonian stationary surfaces (also called $H$-minimal surfaces, see~\cite{Oh2, Oh3}). 
\hfill $\Box$
\end{Rm}
\subsection{Conformal invariance}
Let $\om$ be a domain bi-lipschitz homeomorphic to a disc and let $p\in \om$. We consider the Green function $G_p^\om$ satisfying
\be
\label{0-I.1}
\lf\{
\begin{array}{l}
\ds\Delta G_p^\om=\, 2\pi\, \delta_p\quad\mbox{ in }\om\\[5mm]
\ds G_p^\om=0\quad\mbox{ on }\p \om\ .
\end{array}
\rg.
\ee
$G_p^\om$ is the unique function such that
\be
\label{0-I.2}
\forall \ f\in C^2_0(\om) \quad\mbox{ there holds }\quad\int_\om  G_p^\om\ \ov{\p}\p f=\frac{\pi}{i}\, f(p)\ .
\ee
Let $\phi$ be an holomorphic transformation from $D^2$ into $\om$ sending $0$ to $p$. There holds obviously
\be
\label{0-I.3}
\forall \ f\in C^2_0(\om) \quad\mbox{ there holds }\quad\int_{D^2}  \phi^\ast\lf[G_p^\om\ \ov{\p}\p f\rg]=\int_\om  G_p^\om\ \ov{\p}\p f=\frac{\pi}{i}\, f(p)\ .
\ee
We have also obviously
\be
\label{0-I.4}
\quad\int_{D^2}  \phi^\ast\lf[G_p^\om\ \ov{\p}\p f\rg]=\int_{D^2}  G_p^\om\circ\phi\ \ov{\p}\p \lf(f\circ\phi\rg)=\frac{\pi}{i}\, f\circ\phi(0)
\ee
By uniqueness of the Green function we deduce
\be
\label{0-I.5}
 \forall x\in D^2\quad \quad G_p^\om\circ\phi(x)=2\pi\,\log |x|\ .
\ee
Observe moreover that the system of equations
\be
\label{0-I.0}
\lf\{
\begin{array}{l}
\operatorname{div}(\,g\,\nabla u)=0\quad\mbox{ in }\om\\[5mm]
\ds\operatorname{div}(\,\ov{g}\,\nabla g)=0\quad\mbox{ in }\om\\[5mm]
\ds\ov{g}\,\p_\nu g=0 \quad\mbox{ on }\p \om\
\end{array}
\rg.
\ee
is invariant under the composition with $\phi$.

\subsection{Admissible $S^1$ harmonic maps}
Observe first that there is a one to one correspondence between$S^1$ valued harmonic maps $g$ in $D^2$ having isolated points singularities $p_l$ with degree $d_l$ and satisfying
\[
\begin{array}{l}
\ds\operatorname{div}(\,\ov{g}\,\nabla g)=0\quad\mbox{ in }D^2\\[5mm]
\ds\p_\nu g=0\quad\mbox{ on }\p D^2
\end{array}
\]
and Green functions satisfying
\be
\label{II.0001}
\lf\{
\begin{array}{rl}
\ds\Delta G&\ds=2\pi\ \sum_{l=1}^N\ d_l\, \delta_{p_l}\quad\mbox{in }{\mathcal D}'(D^2)\\[5mm]
\ds G&\ds=0\quad\mbox{ on }\quad\p D^2\ .
\end{array}
\rg.
\ee
This correspondence is given by the equation
\be
\label{0-I-6}
\ov{g}\,\nabla g=i\,\nabla^\perp G\ .
\ee

\begin{Dfi}
\label{df-I.1}
Let $(d_l)_{l=1\cdots N}$ such that $d_l=\pm1$ and $\sum_{l=1}^N d_l=0$. Let $p_l$ be $N$ distinct points in $D^2$ and let $G$ be the Green's function solution to 
\be
\label{0-I-7}
\lf\{
\begin{array}{rl}
\ds\Delta G&\ds=2\pi\ \sum_{l=1}^N\ d_l\, \delta_{p_l}\quad\mbox{in }{\mathcal D}'(D^2)\\[5mm]
\ds G&\ds=0\quad\mbox{ on }\quad\p D^2\ .
\end{array}
\rg.
\ee
We say that $G$ or its associated $S^1-$harmonic map is admissible if each of the connected components of $G\ne 0$ is a disc containing exactly one $p_l$.\hfill $\Box$
 \end{Dfi}
 Observe that for any admissible Green Function $G$ and for ${\mathcal H}^1-$almost every $t\in{\R}$ each of the connected components $\om_i(t)$ of $\{ G(x)\le t\}$  satisfy
\be
\label{0-I-8}
\int_{\p \om_j(t)}\frac{\p G}{\p \nu}=2\pi\ ,
\ee
 Because of the conformal invariance of the Green function mentioned in the
 previous subsection, for each connected component $\om$ of $G\ne 0$ there exists an holomorphic  transformation $\phi$ from $D^2$ into $\om$ sending $0$ to $p_l$ such that
 \be
\label{0-I.9}
 \forall x\in D^2\quad \quad G\circ\phi(x)=2\pi\,\log |x|\ .
\ee
This implies in particular that every $t\in {\R}\setminus \{0\}$ is a regular point for $G$.
\section{Everywhere discontinuous Hamiltonian stationary discs}
\label{sec:everywhere-discontinuous}
\reset
\begin{Th}
\label{th-every}
For any $1\le p<2$ there exists $(\Phi,g)\in W^{1,p}(D^2,{\C}^2)\times W^{1,p}(D^2,S^1)$ solution to the Conformal Hamiltonian Stationary System
\be
\label{E.1}
\lf\{
\begin{array}{l}
\mbox{ div}(g\,\nabla \Phi)=0\quad\mbox{in }\ {\mathcal D}'(D^2)\ \\[5mm]
\ds |\p_{x_1}\Phi|^2=|\p_{x_2}\Phi|^2\quad\mbox{ a. e. in } D^2\\[5mm]
\ds \p_{x_1}\Phi\cdot\p_{x_2}\Phi=0\quad\mbox{ a. e.  in }D^2\\[5mm]
\ds \operatorname{div}\lf[ g^{-1}\,\nabla g\rg]=0\quad\mbox{in }{\mathcal D}'(D^2)\\[5mm]

\end{array}
\rg.
\ee
and $\Phi$ as well as $g$ are nowhere continuous.\hfill $\Box$
\end{Th}
\noindent{\bf Proof of theorem~\ref{th-every}} Let $g$ be the everywhere discontinuous harmonic maps constructed in \cite{Al}(and also \cite{Ji}). The map $\Phi:=(\ov{g}, iG)$ where
$g^{-1}\nabla g=i\,\nabla^\perp G$ is solving the theorem.\hfill $\Box$
\begin{Rm}
\label{rm-every} Observe that $\Phi$ realises a conformal infinite covering of the cylinder $S^1\times\{0\}\times{\R}$. \hfill $\Box$
\end{Rm}

\section{The case of one point singularity}
\reset
\label{sec:one-point}
In this section we prove the main existence result in a simple configuration: a disc with only a singularity of degree $1$ (or $-1$). In this particular case, solutions can also be constructed by Fourier decomposition (see Remark \ref{Rmk: solutions-with-Fourier}). We present instead a more indirect proof, which seems to be more robust and provides a more delicate analysis following a convergence procedure reminiscent of that introduced in~\cite{BBH}. We will generalize this approach in the next section to prove the existence result in the case of several singularities.
\begin{Th}
\label{thm-one-singularity}
Let $g\in W^{1,(2,\infty)}(D^2, S^1)$ weakly harmonic with exactly one point singularity $p\in D^2\subset {\C}$ of multiplicity $\pm 1$ satisfying
\be
\label{I.0}
\lf\{
\begin{array}{l}
\ds\operatorname{div}(\,\ov{g}\,\nabla g)=0\quad\mbox{ in }D^2\\[5mm]
\ds\ov{g}\,\p_\nu g=0 \quad\mbox{ on }\p D^2\
\end{array}
\rg.
\ee
Let $\psi\in C^1(\p D^2,{\C})$. Then there exists  $u\in W^{1,2}(D^2,{\C})\cap C^\infty_{loc}(D^2\setminus\{p\},{\C})$ satisfying
\be
\label{I.00}
\lf\{
\begin{array}{l}
\operatorname{div}(\,g\,\nabla u)=0\quad\mbox{ in }D^2\\[5mm]
\ds u=\psi\quad\mbox{ on }\p D^2\ .
\end{array}
\rg.
\ee
if and only if
\be
\label{I.000}
\int_{\p D^2}\psi\, d{g}=0\ .
\ee
Moreover, if (\ref{I.000}) holds $u$ is unique.
\hfill $\Box$
\end{Th}
\begin{Rm}
\label{rm-I.1} One easily sees that (\ref{I.000}) is a necessary condition. Indeed, introducing the Green function $G$ equal to zero on $\p D^2$ such that
\[
\ov{g}\,\nabla g=i\,\nabla^\perp G\ ,
\]
there holds
\[
\int_{\p D^2} u\, dg=\int_{D^2} \nabla u\cdot \nabla^\perp g\ dx^2=-i\,\int_{D^2} g\,\nabla u\cdot \nabla G\ dx^2=i\ \int_{D^2}\,\operatorname{div}(\,g\,\nabla u)\, G\, dx^2=0
\]
\end{Rm}
\begin{Co}
\label{co-I.0}
Let $g\in W^{1,(2,\infty)}(D^2, S^1)$ weakly harmonic with exactly one point singularity $p\in D^2$ of multiplicity $\pm 1$ satisfying
\be
\label{0.0}
\lf\{
\begin{array}{l}
\ds\operatorname{div}(\,\ov{g}\,\nabla g)=0\quad\mbox{ in }D^2\\[5mm]
\ds\ov{g}\,\p_\nu g=0 \quad\mbox{ on }\p D^2\
\end{array}
\rg.
\ee
Let $\psi\in C^1(\p D^2,{\C})$ such that
\be
\label{I.0000}
\int_{\p D^2}\psi\, dg=0\ .
\ee
Then there exists a unique $\Phi\in W^{1,2}(D^2,{\C}^2)\cap C^\infty_{loc}(D^2\setminus\{p\},{\C}^2)$  weakly conformal on $D^2$ with exactly one conical point singularity satisfying
\be
\label{0.00}
\lf\{
\begin{array}{l}
\operatorname{div}(\,g\,\nabla \Phi)=0\quad\mbox{ in }D^2\\[5mm]
\ds u=\psi\quad\mbox{ on }\p D^2\
\end{array}
\rg.
\ee
where $\Phi:=(u,v)$.
\begin{proof}
This is deduced from Theorem~\ref{thm-one-singularity} by taking  the map $u$ given by the theorem~\ref{thm-one-singularity} and $v$ satisfying $g\nabla u=\nabla^\perp\ov{v}$.
\end{proof}
\end{Co}

\medskip
First of all, we may assume w.l.o.g.~that the multiplicity of the singularity at $p$ is $+1$. By the conformal invariance mentioned in the preliminaries, by composing with the M\"obius transformation
\[
\phi(z)=\frac{z-p}{1-\ov{p}\,z},
\]
we can move the point $p$ to the origin and consider from now on $p=0$ and $g=e^{i\theta}$.
Let $G$ be a real valued  function such that
\be
\label{I.1}
\ov{g}\,\nabla g=i\,\nabla^\perp G\ .
\ee
We assume that $G\equiv 0$ on $\p D^2$. Hence $G$ satisfies in particular
\be
\label{I.2}
\lf\{
\begin{array}{l}
\ds\Delta G=\, 2\pi\, \delta_0\quad\mbox{ in }D^2\\[5mm]
\ds G=0\quad\mbox{ on }\p D^2\ .
\end{array}
\rg.
\ee
Classical results give 
\be
\label{I.3}
G(x)=\log |x|\ .
\ee 
We consider $\psi\in H^\frac{1}{2}(\partial D^2,{\C})$ and we are looking for a solution $u\in W^{1,2}(D^2,{\C})$ satisfying
\be
\label{I.4}
\lf\{
\begin{array}{l}
\ds\operatorname{div}(g\ \nabla u)=0\quad\mbox{ in }D^2\\[5mm]
\ds u=\psi\quad\mbox{ on }\p D^2\ .
\end{array}
\rg.
\ee
To that end, we introduce the following family of variational problems for $t\in [0,1)$:
\be
\label{I.5}
E_t(\psi):=\inf\lf\{
\begin{array}{c}
\ds L_t(u):=\frac{1}{2}\int_{D^2}|\nabla u|^2\ dx^2+\frac{t}{2}\int_{D^2}G\, \lf<i\,\nabla u,\nabla^\perp u\rg>\ dx^2\\[5mm]
\ds u\in W^{1,2}(D^2,{\C})\quad\mbox{ and }\quad u=\psi\quad\mbox{ on }\p D^2
 \end{array} \rg\}.
\ee
We have the first lemma
\begin{Lm}
\label{lm-E-t-psi}
We have
\be
\label{I.6}
\forall t\in [0,1)\quad E_t(\psi)>0
\ee
and $E_t(\psi)$ is achieved by a unique minimizer $u_t\in W^{1,2}_\psi(D^2,{\C})$ which satisfies the following Euler--Lagrange equation
\be
\label{I.7}
\lf\{
\begin{array}{l}
\ds-\Delta u_t\ds=t\ i\, \nabla^\perp G\cdot \nabla u_t\quad\mbox{ in }D^2\\[5mm]
\ds u_t\ds=\psi\quad\mbox{ on }\p D^2.
\end{array}
\rg.
\ee
\hfill $\Box$
\end{Lm}
\noindent{\bf Proof of lemma~\ref{lm-E-t-psi}.} For any $u\in W^{1,2}(D^2,{\C})$ we consider $b$  to be the unique solution in $W^{1,2}_0(D^2,{\C})$ of
\be
\label{I.8}
\lf\{
\begin{array}{l}
\ds\Delta b\ds=\lf<i\,\nabla u,\nabla^\perp u\rg>\quad\mbox{ in }D^2\\[5mm]
\ds b=0\quad\mbox{ on }\p D^2.
\end{array}
\rg.
\ee
Observe that 
\be
\label{I.9}
\ds\lf<i\,\nabla u,\nabla^\perp u\rg>= -2\,\p_{x_1}u_1\,\p_{x_2}u_2+2\,\p_{x_2}u_1\,\p_{x_1}u_2=-2\,\mbox{det}(\nabla u).
\ee
The main result in \cite{Top} gives
\be
\label{I.10}
\|b\|_{L^\infty(D^2)}\le \frac{1}{2\pi}\,\int_{D^2}|\nabla u|^2\ dx^2
\ee
Moreover since both $G$ and $b$ are equal to zero on $\p D^2$ we have
\be
\label{I.11}
\int_{D^2}G\, \lf<i\,\nabla u,\nabla^\perp u\rg>\ dx^2=\int_{D^2}b\ \Delta G\ dx^2=2\pi\ b(0).
\ee
Combining (\ref{I.10}) and (\ref{I.11}) gives in particular
\be
\label{I.12}
\lf|\int_{D^2}G\, \lf<i\,\nabla u,\nabla^\perp u\rg>\ dx^2\rg|\le \int_{D^2}|\nabla u|^2\ dx^2,
\ee
from which we deduce for any $u\in W^{1,2}_\psi(D^2,{\C})$
\be
\label{I.13}
 L_t(u):=\frac{1}{2}\int_{D^2}|\nabla u|^2\ dx^2+\frac{t}{2}\int_{D^2}G\, \lf<i\,\nabla u,\nabla^\perp u\rg>\ dx^2\ge \frac{(1-t)}{2}\,\int_{D^2}|\nabla u|^2\ dx^2\ .
\ee
Hence $ L_t$ is bounded from below and coercive in $W^{1,2}_\psi(D^2,{\C})$. Since $L_t$ is strictly convex on $W^{1,2}_\psi(D^2,{\C})$ there exists a unique minimizer. From classical calculus of variations
we obtain the Euler--Lagrange equation \eqref{I.7}. The lemma is proved. \hfill $\Box$

\medskip

\noindent{\bf Proof of Theorem~\ref{thm-one-singularity}.} 
Assume
\be
\label{I.14}
\limsup_{t\rightarrow 1}\int_{D^2}|\nabla u_t|^2\ dx^2=+\infty
\ee
Hence there exists a sequence $t_k<1$ with $t_k\rightarrow 1$ such that, if we denote $u_k:=u_{t_k}$, 
\be
\label{I.15}
\lim_{k\rightarrow +\infty}\int_{D^2}|\nabla u_k|^2\ dx^2=+\infty
\ee
Let $v_k$ given by
\be
\label{I.16}
v_k:=\frac{u_k}{\|\nabla u_k\|_{L^2(D^2)}}.
\ee
We will be done if we can show that this is incompatible with the fact that the boundary values of $v_k$ converge to zero. This is precisely the content of the following Lemma, which finishes the proof of Theorem~\ref{thm-one-singularity}. \qed

\begin{Lm}
\label{lm:main-lemma-disk}
Let $v_k$ be a sequence of solutions to \eqref{I.7} with $t_k \nearrow 1$ such that
\begin{equation}
\label{eq: convergence-traces-H12}
    \| v_k \|_{H^{1/2}(\partial D^2)} \xrightarrow{k \to \infty} 0.
\end{equation}
Then $\| \nabla v_k \|_{L^2(D^2)} \to 0$.
\end{Lm}
\begin{Rm}
    An analogous result for the $L^\infty$ norm is given in Lemma \ref{lem: vanishing-Linfty-norm}.
\end{Rm}
\noindent{\bf Proof of Lemma~\ref{lm:main-lemma-disk}.} 
We argue by contradiction and suppose that $\| \nabla v_k \|_{L^2(D^2)}$ is bounded away from zero along a subsequence. We can normalize and assume in fact that $\| \nabla v_k \|_{L^2(D^2)} = 1$, since this will increase the $H^{1/2}$ norms of the boundary data by at most a constant factor.
 We extract a subsequence still denoted $v_k$ such that $v_k$ weakly converges in $W^{1,2}(D^2,{\C})$ to a limit $v_\infty$.





Let $c_k$ be the solution of
\begin{align}
    \begin{cases}
        \Delta c_k=\langle i\nabla^\perp v_k, \nabla v_k\rangle=\operatorname{div}(\langle i\nabla^\perp v_k, v_k\rangle)&\text{ in }D^2\\
        c_k=0&\text{ on }\partial D^2.
    \end{cases}
\end{align}
Integrating the equation satisfied by $v_k$ against $v_k$ we obtain
\be
\label{I.47}
-\,\int_{D^2}v_k\cdot\Delta v_k\ dx^2=t_k\, \int_{D^2}\langle i\,v_k,\nabla v_k\rangle\cdot \nabla^\perp G dx^2=t_k\int_{D^2}G\langle i\nabla^\perp v_k, \nabla v_k\rangle.
\ee
This gives
\be
\label{I.48}
\begin{array}{l}
\ds-\,\int_{\p D^2} v_k\cdot\p_\nu v_k\ d\theta+\int_{D^2}|\nabla v_k|^2\ dx^2=\,t_k\, \int_{D^2} G\ \operatorname{div}\langle i\,v_k,\nabla^\perp v_k\rangle\ dx^2\\[5mm]
\ds=\,t_k\, \int_{D^2} G\ \Delta c_k\ dx^2=\,t_k\, 2\pi\, c_k(0)
\end{array}
\ee
Now we notice that $\partial_\nu v_k$ is uniformly bounded in $H^{-\frac{1}{2}}(\partial D^2)$: to see this let $\varphi\in H^\frac{1}{2}(\partial D^2)$ and let $\tilde{\varphi}$ be an extension of $\varphi$ in $W^{1,2}(D^2)$ supported on $\overline{D^2} \setminus D_{1/2}^2$ which has positive distance from the origin, and such that $\lVert \tilde{\varphi}\rVert_{W^{1,2}(D^2)}\leq K\lVert \varphi\rVert_{H^\frac{1}{2}(\partial D^2)}$ (notice that $K$ and $N$ can be chosen  independently from $\varphi$). Then we have
\begin{align}
    \int_{\partial D^2}\varphi\partial_\nu v_k=&\int_{D^2}\operatorname{div}(\tilde{\varphi}\nabla v_k)=\int_{D^2}\nabla\tilde{\varphi}\nabla v_k+\tilde{\varphi}\Delta v_k\\
    \nonumber
    =&\int_{D^2}\nabla\tilde{\varphi}\nabla v_k-it_k\tilde{\varphi}\nabla^\perp G\nabla v_k\leq C\lVert \tilde{\varphi}\rVert_{W^{1,2}(D^2)}\lVert \nabla v_k\rVert_{L^2(D^2)}\\
    \nonumber
    \leq&CK\lVert \varphi\rVert_{H^\frac{1}{2}(\partial D^2)}\lVert \nabla v_k\rVert_{L^2(D^2)}\ ,
\end{align}
so that
\begin{align}
    \lVert \partial_\nu v_k\rVert_{H^{-\frac{1}{2}}(\partial D^2)}\leq C \lVert \nabla v_k\rVert_{L^2(D^2)}.
\end{align}
Since by construction $\lVert \nabla v_k\rVert_{L^2(D^2)}=1$ for any $k\in \mathbb{Z}$, we conclude that $\partial_\nu v_k$ is uniformly bounded in $H^{-\frac{1}{2}}(\partial D^2)$.\\
Together with assumption \eqref{eq: convergence-traces-H12} this implies that the first term in \eqref{I.48} converges to zero. We deduce that
\be
\label{I.52}
\lim_{k\to\infty}\left(\int_{D^2}|\nabla v_k|^2 \ dx^2- 2\pi\, c_k(0)\right)=0\ .
\ee
We have
\be
\label{I.54}
\begin{array}{rl}
\ds{c}_k(0)&\ds=\frac{1}{2\pi}\int_{0}^{1}\int_0^{2\pi}\log r\ \lf[\frac{1}{r}\lf<i\,\frac{\p v_k}{\p r},\frac{\p v_k}{\p \theta} \rg>-\frac{1}{r}\lf<i\,\frac{\p {v}_k}{\p \theta},\frac{\p v_k}{\p r}\rg>\rg]\ r\, dr\, d\theta\\[5mm]
\ds\ &\ds=\frac{1}{2\pi}\int_{0}^{1}\int_0^{2\pi}\log r\ \lf[\frac{\p}{\p r}\lf<i\, v_k,\frac{\p v_k}{\p \theta} \rg>-\frac{\p}{\p\theta}\lf<i\, v_k,\frac{\p v_k}{\p r}\rg> \rg]\  dr\, d\theta\\[5mm]
\ds\ &\ds=\frac{-1}{2\pi}\int_{0}^{1}\int_0^{2\pi}\ \lf<i\, \frac{v_k-\ov{v}_k}{r},\frac{\p v_k}{\p \theta} \rg>  dr\, d\theta
\end{array}
\ee
We decompose $v_k$ in Fourier
\[
v_k=\ov{v}_k+\sum_{j\ne 0} v_k^j\ 
\]
where $v_k^j(r,\theta)=w_k^j(r)\ e^{ij\theta}$, and (\ref{I.54}) becomes
\be
\label{I.55}
\begin{array}{rl}
\ds{c}_k(0)&\ds=-\frac{1}{2\pi}\int_{0}^{1}\sum_{j\ne 0}\int_0^{2\pi}\ \lf<i\, \frac{v^j_k}{r},\frac{\p v^j_k}{\p \theta} \rg>   d\theta\, dr\ .
\end{array}
\ee
Notice that for any $j\in \mathbb{Z}$ we have $\frac{\partial v_k^j}{\partial \theta}= ij v_k^j$. Therefore there holds
\begin{align}
    \int_0^{2\pi}\left\langle i \frac{v_k^j}{r},\frac{\partial v_k^j}{\partial \theta}\right\rangle d\theta=\frac{1}{j}\frac{1}{r}\int_0^{2\pi}\left\lvert \frac{\partial v_k^j}{\partial\theta}\right\rvert^2d\theta   
\end{align}
Hence
\be
\label{I.57}
{c}_k(0)=-\frac{1}{2\pi}\sum_{j\in {\Z}^\ast}\frac{1}{j}\, \int_0^{+\infty}\int_0^{1}\frac{1}{r^2}\,\lf|\frac{\p v^j_k}{\p\theta}\rg|^2\ d\theta \, r\, dr\ .
\ee
On the other hand we have
\be
\label{I.58}
\int_{D^2}|\nabla v_k|^2\ dx^2=\int_{D^2}|\nabla \ov{v}_k|^2\ dx^2+\sum_{j\in {\Z}^\ast}\, \int_0^{1}\int_0^{2\pi}\lf[\lf|\frac{\p v^j_k}{\p r}  \rg|^2+\frac{1}{r^2}\,\lf|\frac{\p v^j_k}{\p\theta}\rg|^2\rg]\ d\theta \, r\, dr\ .
\ee
Combining (\ref{I.52}), (\ref{I.57}) and (\ref{I.58}) is giving
\be
\label{I.59}
\limsup_{k\rightarrow +\infty}\int_{D^2}|\nabla \ov{v}_k|^2\ dx^2+\sum_{j\in {\Z}^\ast}\, \int_0^{1}\int_0^{2\pi}\lf[\lf|\frac{\p v^j_k}{\p r}  \rg|^2+\lf(1+\frac{1}{j}\rg)\frac{1}{r^2}\,\lf|\frac{\p v^j_k}{\p\theta}\rg|^2\rg]\ d\theta \, r\, dr\le 0.
\ee
We deduce
\be
\label{I.60}
\lim_{k\rightarrow +\infty}\int_{D^2}\lf|\frac{\p v_k}{\p r}  \rg|^2\ dx^2=0\quad\mbox{ and }\quad \lim_{k\rightarrow +\infty}\int_{D^2}\frac{1}{r^2}\,\lf|\frac{\p v_k}{\p \theta} - \frac{\p v^{-1}_k}{\p \theta} \rg|^2\ dx^2=0.
\ee
Next we claim that if assumption \eqref{I.000} holds, then $ v_k^{-1} \equiv 0$.\\
Let $v_k^-:=v_k^{-1}$ and $w_k^-:=w_k^{-1}$. By assumption we have 
\begin{align}
\label{I.64}
    \ds\ddot{w}^-_k+r^{-1}\,\dot{w}^-_k-r^{-2}\,(1-t_k)\,w^-_k=0
\end{align}
We introduce the change of variables $r=\exp(s)$ for $s\in [-\infty,0]$ and  $Y_k(s):=w^-_k(\exp(s))$. We have \[
\dot{Y}_k(s)= r\, \dot{w}_k^-(r)\quad\mbox{ and }\quad\ddot{Y}_k(s)= r^2\ \ddot{w}_k^-(r)+ r\, \dot{w}^-_k\ .
\]
Hence (\ref{I.64}) becomes
\be
\label{I.65}
\ds\ddot{Y}_k(s)-(1-t_k)\, Y_k(s)=0.
\ee
In particular we have
\[
Y_k(s)= A_k\ e^{\sqrt{1- t_k}\, s}+B_k\ e^{-\sqrt{1- t_k}\, s}\ .
\]
since $v_k\in W^{1,2}(D^2)$, we have $v_k^{-1}\in W^{1,2}$, therefore $B_k=0$ for any $k$.
Now \eqref{I.000} implies that
\be
\label{I.67}
\int_{\p D^2}\ v_k(1, \theta)\ e^{ i\theta}\ d\theta=0\quad ,
\ee
therefore $Y_k(0)=0$, and we conclude that $Y_k\equiv 0$, so that $v_k^{-1}\equiv 0$. This proves the claim.\\
Thus (\ref{I.60}) implies that if \eqref{I.000} holds, then $v_k$ converges strongly to $0$ in $W^{1,2}(D^2)$. This contradicts
\be
\label{I.68}
\|\nabla v_k\|_{L^2}=1\ ,
\ee
as desired.
Uniqueness follows from Lemma \ref{lma-unique}.\\
The converse implication has been discussed in Remark \ref{rm-I.1}. 
\hfill $\Box$

\begin{Rm}
\label{Rmk: solutions-with-Fourier}
    We can construct more explicit solutions to \eqref{I.00} by means of Fourier decomposition. Assume fist that $p=0$, and that the singularity has degree $+1$.
    Write the boundary datum $\psi$ as a sum of its Fourier modes:
    \begin{align}
        \psi(\theta)=\sum_{k\in \mathbb{Z}}\hat{\psi}(k)e^{ik\theta}.
    \end{align}
    Assumption \eqref{I.000} implies that $\hat{\psi}(-1)=0$. Then the unique solution of \eqref{I.00} in $W^{1,2}(D^2)$ is given by
    \begin{align}
        u(r, \theta)=\sum_{\substack{k\in \mathbb{Z}\\ k\neq -1}}\hat{\psi}(k)r^{\sqrt{k(k+1)}}e^{ik\theta}.
    \end{align}
    The case of degree $-1$ can be obtained by conjugation.
    Finally if $p\neq 0$, let $\Phi$ be a conformal diffeomorphism of $D^2$ mapping $p$ to $0$. The construction above gives a representation of $u\circ\Phi$ in terms of the Fourier coefficients of $\psi\circ\Phi$.
\end{Rm}

\section{The general case}
\label{sec:general-case}
\reset
\subsection{Preliminary observations}
Let $g$ be an harmonic map into $S^1$ with isolated singularities and $G$ such that
\[
\ov{g}\,\nabla g=i\,\nabla^\perp G
\]
We consider functions $f$ such that
\be
\label{II-01}
d f\wedge dG=0\ .
\ee
If $x$ is a regular point for $G$ we have locally $f(x)=F(G(x))$. Let $u_j:= f_j(x)\ g^j$ where $f_j$ satisfies (\ref{II-01}). We compute around a regular point of $G$
\be
\label{II-02}
\begin{array}{l}
\ds\Delta u_j=\Delta\lf(F_j(G)\,g^j\rg)=\Delta\lf(F_j(G)\rg)\,g^j+2\,\nabla \lf(F_j(G)\rg)\cdot\nabla g^j+F_j(G)\,\Delta g^j\\[5mm]
\ds\quad=\ddot{F}_j(G)\,|\nabla G|^2\,g^j-j^2\, F_j(G)\, g^j\,|\nabla G|^2
\end{array}
\ee
where we used respectively $\Delta G=0$ at that point, $\nabla G\cdot\nabla g=0$, $|\nabla g|=|\nabla G|$ and $g^k$ is harmonic into $S^1$ and satisfies
\be
\label{II-03}
-\Delta g^j= j^2\,g^j \,|\nabla g|^2\ .
\ee
From (\ref{II-02}) we obtain for $t\in{\R}$
\be
\label{II-04}
\begin{array}{rl}
\ds\Delta u_j+i\,t\,\nabla^\perp G\cdot\nabla u_j&\ds=\ddot{F}_j(G)\,|\nabla G|^2\,g^j-j^2\, F_j(G)\, g^j\,|\nabla G|^2+i\,j\,t\, g^{j-1}\, F_j(G)\,\nabla^\perp G\cdot\nabla g\\[5mm]
\ds\quad&\ds=\ddot{F}_j(G)\,|\nabla G|^2\,g^j-j^2\, F_j(G)\, g^j\,|\nabla G|^2+j\,t\, g^{j-1}\, F_j(G)\,\ov{g}\,\nabla g\cdot\nabla g\\[5mm]
\ds\quad&\ds=\ddot{F}_j(G)\,|\nabla G|^2\,g^j-j^2\, F_j(G)\, g^j\,|\nabla G|^2-j\,t\, g^{j}\, F_j(G)\,\nabla \ov{g}\cdot\nabla g\\[5mm]
\ds\quad&\ds=\lf[ \ddot{F}_j(G) -(j^2+t\,j)\, F_j(G)  \rg]\,|\nabla G|^2\,g^j
\end{array}
\ee
We have obtained the following lemma
\begin{Lm}
\label{lm-II-01}
Let $u_j(x):= F_j(G(x))\ g^j(x)$ where $g$ is an $S^1$ harmonic map with isolated singularities and 
\[
\ov{g}\,\nabla g=i\,\nabla^\perp G\ .
\]
Away from the singularities and at the regular points of $G$, $u_j$ solves
\be
\label{II.05}
\Delta u_j+i\,t\,\nabla^\perp G\cdot\nabla u_j=0
\ee
for $t\in (0,1)$ if and only if
\be
\label{II.06}
F_j(s)= A_j \,e^{\sqrt{j^2+t\,j}\,s}+B_j \,e^{-\sqrt{j^2+t\,j}\,s}
\ee
\hfill $\Box$
\end{Lm}
\begin{Rm}
\label{rm-II-01} The previous lemma is a generalisation of the Fourier decomposition of the solutions in the axially symmetric case of the previous section (see Remark \eqref{Rmk: solutions-with-Fourier} and (\ref{I.65}) for $j= -1$). 
\end{Rm}
From now on we shall be considering the following setting.
Let $(d_l)_{l=1\cdots N}$ such that $d_l=\pm1$. Let $p_i$ be $N$ distinct points in $D^2$ and let $G$ be the Green's function solution to 
\be
\label{II.001}
\lf\{
\begin{array}{rl}
\ds\Delta G&\ds=2\pi\ \sum_{l=1}^N\ d_l\, \delta_{p_l}\quad\mbox{in }{\mathcal D}'(D^2)\\[5mm]
\ds G&\ds=0\quad\mbox{ on }\quad\p D^2\ .
\end{array}
\rg.
\ee
Assume that for ${\mathcal H}^1-$almost every $t\in{\R}$ each of the connected components $\om_i(t)$ of $\{ G(x)\le t\}$  satisfy
\be
\label{II.002}
\int_{\p \om_j(t)}\frac{\p G}{\p \nu}=2\pi\ ,
\ee
 and each connected component of $G\ne 0$ is a disc, that is, according to our previous definition, $G$ is {\it admissible}. We also choose $\psi\in H^\frac{1}{2}(\partial D^2,{\C})$.\\
 
 We introduce the following space
\[
\mathfrak{H}:=\lf\{
\begin{array}{l}
 u\in W^{1,(2,\infty)}(D^2,{\C})\ ;\ \mbox{ on each connected component $\om$ of $G\ne 0$ }\\[3mm]
 \exists\  \ti{u}\ \in W^{1,2}(\om,{\C})\ \mbox{  and }\ A^\om\in {\C}\quad\mbox{ s.t. }u=\ti{u}+ A^\om\,g^{-1}
\end{array}
\rg\}
\]
 \begin{Lm}
\label{lm-exp}
Let $t\in [0,1]$. Let $u_t\in W^{1, 2}(D^2)$ ($u_1\in \mathfrak{H}$ if $t=1$) be a solution to
\be
\label{II.34}
\begin{array}{rl}
\ds-\Delta u_t&\ds=t\ i\, \nabla^\perp G\cdot \nabla u_t\quad\mbox{ in }D^2
\end{array}
\ee
where $G$ satisfies the assumptions above.
Then $u_t$ satisfies
\be
\label{II.35}
\begin{array}{l}
\ds u_t(x)=\lf\{
\begin{array}{l}
\ds\sum_{j\in \Z} A^\om_j \ e^{-\,\sqrt{j^2+t\,j}\, G}\, g^j\quad\mbox{ on every connected component $\om$ where $G>0$}\\[5mm]
\ds\sum_{j\in \Z} A^\om_j \ e^{\sqrt{j^2+t\,j}\, G}\, g^j\quad\mbox{ on every connected component $\om$ where $G<0$}\
\end{array}\rg.
\end{array}
\ee
\hfill $\Box$
\end{Lm}
\noindent{\bf Proof of lemma~\ref{lm-exp}.}
Let $\om_+$ be a connected component of $G>0$.
By the assumptions on $G$, $\omega_+$ is a topological disc with Lipschitz boundary. By the conformal invariance of the problem we see that it is enough to proof the Lemma for the case $\omega_+=D^2$, with a negative singularity at the origin (in fact by the Riemann mapping theorem there exists a conformal map sending $D^2$ to $\omega_+$ and $0$ to $p_{\omega_+}$, the singularity point in $\omega_+$). In this case one can decompose $u_t$ into Fourier frequencies,
\begin{align}
    u_t(r, \theta)=\sum_{k\in \mathbb{Z}}w^k_t(r)e^{i\theta},
\end{align}
and obtain, for any $k\in \mathbb{Z}$, the following ODE
\begin{align}
\label{eq: ODE-coeff}
    \ddot{w}_t^k+r^{-1}\dot{w}^k_t+r^{-2}(t-k)w^k_t=0.
\end{align}
The solutions to \eqref{eq: ODE-coeff} are given by
\begin{align}
    w^k_t(r)=A_kr^{-\sqrt{k^2+tk}}+B_kr^{\sqrt{k^2+tk}}.
\end{align}
Since we are looking for solutions in $W^{1,2}(D^2)$ (or in $\mathfrak{H}$ if $t=1$), we have $B_k=0$ for any $k\in \mathbb{Z}$.
Therefore 
\begin{align}
\label{eq: Fourier-sol}
    u_t(r,\theta)=\sum_{k\in \mathbb{Z}}A_kr^{-\sqrt{k^2+tk}}e^{ik\theta},
\end{align}
where the coefficients $A_k$ are determined by the boundary data (of the component $\omega_+$).
This proves formula \eqref{II.35} for components where $G>0$.
The proof of the formula for the connected components $\omega$ where $G<0$ is analogous.\\
Alternatively, one can deduce the Lemma from Lemma \ref{lm-II-01}, combined with the fact that on every connected component $\omega$, depending on the sign of $G$ either all coefficients $A_j$ or all coefficients $B_j$ must vanish in order for $u_t$ to be in $W^{1,(2,\infty)}(D^2)$.
\hfill $\Box$

\subsubsection{Regularity of solutions}
\begin{Lm}
\label{lm-regularity}
    Let $u \in W^{1,2}(D^2)$ and $g$ be a solution of 
    \be
\label{II.41-a}
\lf\{
\begin{array}{l}
\operatorname{div}(\,g\,\nabla u)=0\quad\mbox{ in }D^2\\[5mm]
\ds\operatorname{div}(\,\ov{g}\,\nabla g)=0\quad\mbox{ in }D^2\\[5mm]
\ds\ov{g}\,\p_\nu g=0 \quad\mbox{ on }\p D^2\ ,
\end{array}
\rg.
\ee
 where $g$ is an  admissible $S^1$-harmonic map on $D^2$ with finitely many singularities. Then $u \in C^{1, \sqrt{2}-1}(D^2_R)$ for any $R\in (0,1)$.\hfill $\Box$
    \end{Lm}
\noindent{\bf Proof of lemma~\ref{lm-regularity}}
        It is clear by standard elliptic estimates that the map $u$ is smooth inside $D^2\setminus\{p_1\cdots p_N\}$. Because of the previous lemma, on each connected component $\om$ of $G\ne 0$ $u$ has the explicit form given by (\ref{II.35}) (with $t=1$).  We assume $G>0$. Thanks to the conformal invariance of the PDE mentioned in the preliminaries, after the composition with  an holomorphic  transformation $\phi$ from $D^2$ into $\om$ sending $0$ to the singularity $p$ we have
 \be
 \label{II.41-b}
 u\circ\phi=\sum_{j\in \mathbb{Z}}a_j\ e^{\sqrt{j^2+j}\,\log r}\ e^{i\,j\, \theta}\ . 
 \ee
Without loss of generality we may assume as well that $a_0 = 0$.  Moreover, since $u\circ\phi\in W^{1,2}(D^2,{\C})$ we must have $a_{-1}=0$. 
Observe that for each $j\in {\Z}$ and $R<1$
        \begin{equation}
            \lVert a_j r^{\sqrt{j(j+1)}} e^{i j \theta} \rVert_{C^{1,\sqrt{2}-1}(B_R(0))}
            \leq C |a_j| |j|^{\sqrt{2}} R^{\sqrt{j(j+1)} - \sqrt{2}}.
        \end{equation}
        Thus we have
        \begin{align*}
            \sum_{j \neq 0, -1} \lVert a_j r^{\sqrt{j(j+1)}} e^{i j \theta} \rVert_{C^{1,\sqrt{2}-1}(B_R(0))}
            &\leq  C \sum_{j \neq 0, -1} |a_j| |j|^{\sqrt{2}} R^{\sqrt{j(j+1)} - \sqrt{2}} \\
            &\leq C \left( \sum_{j \neq 0, 1} |a_j|^2\right)^{1/2} \left( \sum_{j \neq 0, 1} |j|^{2\sqrt{2}} R^{2(\sqrt{j(j+1)} - \sqrt{2})} \right)^{1/2} < +\infty\ ,
        \end{align*}
        since $\sqrt{j(j+1)} \geq \sqrt{2}$ for $j \neq 0, 1$, whenever $R < 1$. Therefore
        \[
            \sum_{j \in \Z} a_j r^{\sqrt{j(j+1)}} e^{i j \theta}
        \]
        is a convergent series in $C^{1, \sqrt{2}-1}(\tilde{B}_R(0))$. The square summability of the sequence $\{ a_j \}$ follows from $\sum_{j \neq 0, 1} |j| |a_j|^2 < +\infty$, which is implied by the fact that  $u \in W^{1,2}(B_1)$.\hfill $\Box$
\subsection{Convergence of $L_t$ minimizers}
The main theorem of this section is the following.  
\begin{Th}
\label{th-II.1}
Let $(d_l)_{l=1\cdots N}$ such that $d_l=\pm1$ and $\sum_{l=1}^N d_l=0$. Let $p_i$ be $N$ distinct points in $D^2$ and let $G$ be the Green's function solution to 
\be
\label{II.1}
\lf\{
\begin{array}{rl}
\ds\Delta G&\ds=2\pi\ \sum_{l=1}^N\ d_l\, \delta_{p_l}\quad\mbox{in }{\mathcal D}'(D^2)\\[5mm]
\ds G&\ds=0\quad\mbox{ on }\quad\p D^2\ .
\end{array}
\rg.
\ee
Assume that for ${\mathcal H}^1-$almost every $t\in{\R}$ each of the connected components $\om_i(t)$ of $\{ G(x)\le t\}$  satisfy
\be
\label{II.2}
\int_{\p \om_j(t)}\frac{\p G}{\p \nu}=2\pi\ ,
\ee
and that every connected component of $G\ne 0$ is a disc. Then for any $\psi\in H^\frac{1}{2}(\p D^2,{\C})$ there exists a unique $u\in \mathfrak{H}$ satisfying
\be
\label{II.4}
\lf\{
\begin{array}{rl}
\operatorname{div}(\,g\,\nabla u)&\ds=0\quad\mbox{ in }D^2\\[5mm]
\ds u&\ds=\psi\quad\mbox{ on }\p D^2\ 
\end{array}
\rg.
\ee
where
\be
\label{II.5}
\ov{g}\,\nabla g=i\,\nabla^\perp G\ .
\ee
Moreover for any connected component $\om$ of $G\ne 0$, the solution $u$ has the form
\be
\label{II.5-a}
u=\sum_{j\in \mathbb{Z}} A_j^{\om}\ g^j\ e^{-\sqrt{j^2+j}\,|G|}\quad\mbox{ and }\quad \ti{u}:= u-A^\om_{-1}\ g^{-1}\in W^{1,2}(\om,{\C})\ .
\ee
\hfill $\Box$
\end{Th}
As for the case of one singularity, we introduce the sequence of variational problems for $t\in [0,1)$:
\be
\label{II.5-aaa}
E_t(\psi):=\inf\lf\{
\begin{array}{c}
\ds L_t(u):=\frac{1}{2}\int_{D^2}|\nabla u|^2\ dx^2+\frac{t}{2}\int_{D^2}G\, \lf<i\,\nabla u,\nabla^\perp u\rg>\ dx^2\\[5mm]
\ds u\in W^{1,2}(D^2,{\C})\quad\mbox{ and }\quad u=\psi\quad\mbox{ on }\p D^2
 \end{array} \rg\}
\ee
Before we start with the proof of the Theorem, we give an an analog to Lemma \ref{lm-E-t-psi} for the present setting.
\begin{Lm}
\label{lm-E-t-psi-II}
We have
\be
\label{II.6}
\forall\ t\in [0,1)\quad E_t(\psi)>0
\ee
and $E_t(\psi)$ is achieved by a unique minimizer $u_t\in W^{1,2}_\psi(D^2,{\C})$ which satisfies the following Euler Lagrange Equation
\be
\label{II.7}
\lf\{
\begin{array}{l}
\ds-\Delta u_t\ds=t\ i\, \nabla^\perp G\cdot \nabla u_t\quad\mbox{ in }D^2\\[5mm]
\ds u_t\ds=\psi\quad\mbox{ on }\p D^2\ .
\end{array}
\rg.
\ee
\hfill $\Box$
\end{Lm}
\noindent{\bf Proof of lemma~\ref{lm-E-t-psi-II}.}
The proof is entirely analogous to the one of Lemma \ref{lm-E-t-psi}.
For any $u\in W^{1,2}(D^2,{\C})$ we consider $b$  to be the unique solution in $W^{1,2}_0(D^2,{\C})$ of
\be
\label{II.8}
\lf\{
\begin{array}{l}
\ds\Delta b\ds=\lf<i\,\nabla u,\nabla^\perp u\rg>\quad\mbox{ in }D^2\\[5mm]
\ds b=0\quad\mbox{ on }\p D^2\ .
\end{array}
\rg.
\ee
Observe that 
\be
\label{II.9}
\ds\lf<i\,\nabla u,\nabla^\perp u\rg>= -2\,\p_{x_1}u_1\,\p_{x_2}u_2+2\,\p_{x_2}u_1\,\p_{x_1}u_2=-2\,\mbox{ det}(\nabla u)\ .
\ee
Because of lemma~\ref{lma-1}
\be
\label{II.10}
\lf|\sum_{j=1}^N d_j \, b(p_j)\rg|\le \frac{1}{2\pi}\,\int_{D^2}|\nabla u|^2\ dx^2
\ee
Moreover since both $G$ and $b$ are equal to zero on $\p D^2$ we have
\be
\label{II.11}
\int_{D^2}G\, \lf<i\,\nabla u,\nabla^\perp u\rg>\ dx^2=\int_{D^2}b\ \Delta G\ dx^2=2\pi\ \sum_{j=1}^N d_j \, b(p_j)\ .
\ee
Combining (\ref{II.10}) and (\ref{II.11}) gives in particular
\be
\label{II.12}
\lf|\int_{D^2}G\, \lf<i\,\nabla u,\nabla^\perp u\rg>\ dx^2\rg|\le \int_{D^2}|\nabla u|^2\ dx^2\ ,
\ee
from which we deduce for any $u\in W^{1,2}_\psi(D^2,{\C})$
\be
\label{II.13}
 L_t(u):=\frac{1}{2}\int_{D^2}|\nabla u|^2\ dx^2+\frac{t}{2}\int_{D^2}G\, \lf<i\,\nabla u,\nabla^\perp u\rg>\ dx^2\ge \frac{(1-t)}{2}\,\int_{D^2}|\nabla u|^2\ dx^2\ .
\ee
Hence $ L_t$ is bounded from below and coercive in $W^{1,2}_\psi(D^2,{\C})$. Since $L_t$ is convex on $W^{1,2}_\psi(D^2,{\C})$ there exists a unique minimizer. From classical calculus of variations
we obtain the Euler Lagrange equation (\ref{II.7}). The lemma is proved. \hfill $\Box$

\noindent{\bf Proof of Lemma~\ref{th-II.1}.} 
We will first prove Theorem \ref{th-II.1} under the assumption that for all connected components $\omega$ of $G\neq 0$ there holds $\lvert\partial \omega\cap\partial D^2\rvert>0$. At the end of the proof we will explain how to remove this assumption.

Assume first that
\be
\label{III.1}
\liminf_{t\rightarrow 1}\int_{D^2}|\nabla u_t|^2\ dx^2<+\infty\ ,
\ee
then a subsequence is converging towards a solution. \\
Assume now that
\be
\label{III.2}
\liminf_{t\rightarrow 1}\int_{D^2}|\nabla u_t|^2\ dx^2=+\infty\ .
\ee
Let $v_k$ given by
\be
\label{II.16}
v_k:=\frac{u_{t_k}}{\|\nabla u_{t_k}\|_{L^2(D^2)}}\ .
\ee
We extract a subsequence (still denoted $v_k$) such that $v_k$ weakly converges in $W^{1,2}(D^2,{\C})$ to a limit $v_\infty$.

\medskip

\noindent{\bf Case 1: $v_\infty\neq 0$}. If this is the case we have by linearity that $v_\infty\in W^{1,2}(D^2,{\C})$ solves

\be
\label{III.17}
\lf\{
\begin{array}{l}
\ds-\Delta v_\infty\ds= i\, \nabla^\perp G\cdot \nabla v_\infty\quad\mbox{ in }D^2\\[5mm]
\ds v_\infty\ds=0\quad\mbox{ on }\p D^2\ .
\end{array}
\rg.
\ee
Multiplying by $v_\infty$ and integrating by parts gives
\be
\label{III.18}
 L_1(v_\infty):=\frac{1}{2}\int_{D^2}|\nabla v_\infty|^2\ dx^2+\frac{1}{2}\int_{D^2}G\, \lf<i\,\nabla v_\infty,\nabla^\perp v_\infty\rg>\ dx^2=0
\ee
Let $c_\infty$ be the solution to
\be
\label{III.19}
\lf\{
\begin{array}{l}
\ds\Delta c_\infty\ds=\lf<i\,\nabla v_\infty,\nabla^\perp v_\infty\rg>\quad\mbox{ in }D^2\\[5mm]
\ds c_\infty=0\quad\mbox{ on }\p D^2\ .
\end{array}
\rg.
\ee
Inserting this identity in (\ref{III.18}) gives
\be
\label{III.20}
\int_{D^2}|\nabla v_\infty|^2\ dx^2=2\pi\, \sum_{j=1}^Nd_j\,c_\infty(p_j)\ .
\ee
This is a case of equality for the lemma~\ref{lma-1} this implies that $v_\infty=0$ which is a contradiction.

\medskip
\noindent{\bf Case 2: $v_\infty= 0$}

In this case we would like to show that
\begin{align}
\label{eq: u-tilde-bounded-W12}
    \liminf_{k\to\infty} \sum_\om\|\nabla \ti{u}_k\|^2_{L^2(\om)}<\infty.
\end{align}
We will show \eqref{eq: u-tilde-bounded-W12} arguing by contradiction.\\
More precisely, we will show that if \eqref{eq: u-tilde-bounded-W12} does not hold, than we will get a contradiction in both of the following cases:
    \begin{enumerate}
        \item $\displaystyle \liminf_{k\to\infty}\frac{\sum_{\om}|A^{\om,k}_{-1}|^2}{\sum_\om\|\nabla \ti{u}_k\|^2_{L^2(\om)}}= 0$,
        \item $\displaystyle \liminf_{k\to\infty}\frac{\sum_{\om}|A^{\om,k}_{-1}|^2}{\sum_\om\|\nabla \ti{u}_k\|^2_{L^2(\om)}}>0$.
    \end{enumerate}
Assume first that
\be
\label{III.23}
\sum_\om\|\nabla \ti{u}_k\|^2_{L^2(\om)}\rightarrow +\infty\quad \mbox{  and  }\quad \liminf_{k\to\infty}\frac{\sum_{\om}|A^{\om,k}_{-1}|^2}{\sum_\om\|\nabla \ti{u}_k\|^2_{L^2(\om)}}= 0
\ee
Then for any $1<p<2$ we have
\be
\label{III.24}
\begin{array}{l}
\ds\int_{D^2}|\nabla u|^p\ dx^2=\sum_\om\int_\om|\nabla u|^p\ dx^2\le C_p\sum_\om\int_\om|\nabla \ti{u}|^p\ dx^2+|A_{-1}^{\om,k}|^p\ \int_\om|\nabla (g^{-1}\, e^{-\sqrt{1-t}\, |G|})|^p\ dx^2\\[5mm]
\ds\quad\le C\, \sum_\om\|\nabla \ti{u}_k\|^p_{L^2(\om)}
\end{array}
\ee
Hence we can extract a subsequence such that 
\be
\label{III.25}
\hat{u}_k:=\frac{u_k}{\sum_\om\|\nabla \ti{u}_k\|_{L^2(\om)}}\rightharpoonup \hat{u}_\infty\quad\mbox{ weakly in }W^{1,p}(D^2).
\ee
Since
\be
\label{III.26}
\operatorname{div}\lf(\nabla \hat{u}_k+i\, t_k\,G\,\nabla^\perp \hat{u}_k    \rg)=0,
\ee
we have at the limit
\be
\label{III.27}
\operatorname{div}\lf(\nabla \hat{u}_\infty+i\, \,G\,\nabla^\perp \hat{u}_\infty    \rg)=0\ .
\ee
On $\om$ we have
\be
\label{III.28}
\hat{u}_k=\frac{\ti{u}_k}{\sum_\om\|\nabla \ti{u}_k\|_{L^2(\om)}}+ \frac{A_{-1}^{\om,k}\, g^{-1}\, e^{-\sqrt{1-t}\, |G|}}{\sum_\om\|\nabla \ti{u}_k\|_{L^2(\om)}}
\ee
Since by assumption (\ref{III.23})
\be
\label{III.29}
\frac{A_{-1}^{\om,k}\, g^{-1}\, e^{-\sqrt{1-t}\, |G|}}{\sum_\om\|\nabla \ti{u}_k\|_{L^2(\om)}}\longrightarrow 0\quad\mbox{ in }{\mathcal D}'(\om)
\ee
and since obviously
\be
\label{III.30}
\lf\|\nabla\lf(\frac{\ti{u}_k}{\sum_\om\|\nabla \ti{u}_k\|_{L^2(\om)}}\rg)\rg\|_{L^2(\om)}\le 1
\ee
we deduce that
\be
\label{III.31}
\sum_{\om}\int_\om|\nabla \hat{u}_\infty|^2\ dx^2\le C
\ee
combined with the fact that $\nabla \hat{u}_\infty$ is absolutely continuous with respect to Lebesgue we deduce that $\hat{u}_\infty$ is in $W^{1,2}(D^2)$. It has a trace equal to zero thanks to the second part of assumption (\ref{III.23}). Therefore we must have $\hat{u}_\infty=0$, otherwise we would have a non trivial case of equality in (\ref{A-3}).\\
Next we claim that the traces of $\frac{\ti{u}_k}{\sum_\om\|\nabla \ti{u}_k\|_{L^2(\om)}}$ on the boundaries of the components $\omega$ go to zero in $H^\frac{1}{2}(\partial \omega)$. By the second assumption in \eqref{III.23} it is enough to prove the statement for the sequence $\hat{u}_k$. Notice that away from $\partial D^2$ and from the singularities of $G$, by a classical bootstrap argument $\hat{u}_k$ is precompact in $C^l$ for any $l\in \mathbb{N}$. Since the weak limit of $\hat{u}_k$ is zero, $\hat{u}_k$ is converging strongly in $C^l$ to zero away from $\partial D^2$ and the singularities of $G$, for any $l\in \mathbb{N}$; therefore it is enough to prove the claim in a neighbourhood of $\partial D^2$.
To this end let $\rho\in (0,1)$ such that all singularities of $G$ are contained in $D^2_\rho$. Set $A=D^2\smallsetminus D^2_\rho$. For any $k\in \mathbb{N}$ let $\hat{u}_k^1$ and $\hat{u}_k^2$ be solutions of
\begin{align}
    \begin{cases}
        -\Delta \hat{u}^1_k=it_k\nabla^\perp G\cdot \nabla \hat{u}_k&\text{ in }A\\
        \hat{u}^1_k=0&\text{ on }\partial A,
    \end{cases}\qquad
    \begin{cases}
        -\Delta \hat{u}^2_k=0 &\text{ in }A\\
        \hat{u}^2_k=\hat{u}_k&\text{ on }\partial A.
    \end{cases}
\end{align}
As observed above, $\hat{u}_k\to 0$ in $H^\frac{1}{2}(\partial D^2_\rho)$. Moreover by the first assumption in \eqref{III.23} we have $\hat{u}_k\to 0$ in $H^\frac{1}{2}(\partial D^2)$. Therefore $\lVert \hat{u}_k^2\rVert_{W^{1,2}(D^2)}\to 0$. In particular the traces of $\hat{u}_k^2$ go to zero in $H^\frac{1}{2}(\partial \omega)$ for any component $\omega$.\\
On the other hand observe that $\nabla G\in L^\infty(A)$, therefore we can apply the following integrability-by-compensation estimate (see for instance Lemma IV.2 in \cite{BerRiv}): for any $p<\infty$
\begin{align}
    \lVert\nabla \hat{u}_k^1\rVert_{L^p(A)}\leq C_p\lVert \nabla \hat{u}_k\rVert_{L^2(A)}\lVert \nabla G\rVert_{L^p(A)}.
\end{align}
As $\lVert \nabla \hat{u}_k\rVert_{L^2(D^2)}$ is uniformly bounded by $1$, $\hat{u}_k^1$ is uniformly bounded in $W^{1,p}(A)$, for some $p\in (2,\infty)$; therefore the sequence converges to zero weakly in $W^{1,p}(A)$. It follows that for any component $\omega$ we have $\hat{u}_k^1\to 0$ weakly in $W^{1-\frac{1}{p},p}(\partial(\omega\cap A))$. Since the embedding $W^{1-\frac{1}{p},p}(\partial(\omega\cap A))\hookrightarrow H^\frac{1}{2}(\partial(\omega\cap A))$ is compact we conclude that $\hat{u}_k^1\to 0$ strongly in $H^{\frac{1}{2}}(\partial (\omega\cap A))$. This concludes the proof of the claim.
Since by definition we have
\begin{align}
    \int_{\partial \omega}\frac{\ti{u}_k}{\sum_\om\|\nabla \ti{u}_k\|_{L^2(\om)}} \, dg=0,
\end{align}
Lemma \ref{lm:main-lemma-disk} implies that
\begin{align}
    \lim_{k\to\infty}\left\lVert \nabla\left(\frac{\ti{u}_k}{\sum_\om\|\nabla \ti{u}_k\|_{L^2(\om)}}\right)\right\rVert_{L^2(\om)}=0.
\end{align}
Since there is at least one $\omega$ such that, for a subsequence,
\begin{align}
\label{III.32}
\lim_{k\rightarrow+\infty}\frac{\|\nabla\ti{u}_k\|_{L^2(\om)}}{\sum_\om\|\nabla \ti{u}_k\|_{L^2(\om)}}>0,
\end{align}
we get a contradiction.\\
Next assume that
\be
\label{III.33}
\liminf_{k\rightarrow +\infty}\sum_\om\|\nabla \ti{u}_k\|^2_{L^2(\om)}\rightarrow +\infty\quad \mbox{  and  }\quad\frac{\sum_{\om}|A^{\om,k}_{-1}|^2}{\sum_\om\|\nabla \ti{u}_k\|^2_{L^2(\om)}}>0\ .
\ee
The second assumption in \eqref{III.33} implies that there is a subsequence such that
\be
\label{III.34}
\lim_{k\rightarrow+\infty}\lf\|\nabla\lf(\frac{u_k}{\sum_{\om}|A^{\om,k}_{-1}|}\rg)\rg\|_{L^{2,\infty}(D^2)}<+\infty\ .
\ee
Moreover \eqref{III.33} also implies that 
\be
\label{III.35}
\mbox{Tr}\lf(\frac{u_k}{\sum_{\om}|A^{\om,k}_{-1}|}\rg)=\frac{\psi}{\sum_{\om}|A^{\om,k}_{-1}|}\longrightarrow 0\quad\mbox{ in }H^\frac{1}{2}(\p D^2)
\ee
We extract a subsequence such that 
\be
\label{III.36}
\check{u}_k:=\frac{u_k}{\sum_{\om}|A^{\om,k}_{-1}|}\rightharpoonup \check{u}_\infty\quad\mbox{ weakly in } W^{1,(2,\infty)}(D^2)
\ee
Observe that for any $\om$
\be
\label{III.37}
\frac{sgn(\om)}{2\pi\,i}\,\int_{\p\om} \check{u}_k\ dg=\frac{A^{\om,k}}{\sum_{\om}|A^{\om,k}_{-1}|}\longrightarrow \frac{sgn(\om)}{2\pi\,i}\,\int_{\p\om} \check{u}_\infty\ dg
\ee
Hence $\check{u}_\infty$ is a non zero solution in $W^{1,(2,\infty)}(D^2)$ of 
\be
\label{III.38}
\lf\{
\begin{array}{rl}
\ds-\Delta \check{u}_\infty\ds&\ds= i\, \nabla^\perp G\cdot \nabla \check{u}_\infty\quad\mbox{ in }D^2\\[5mm]
\ds \check{u}_\infty\ds&\ds=0\quad\mbox{ on }\p D^2\ .
\end{array}
\rg.
\ee
On each connected component $\om$ of $G\ne 0$ in $D^2$ we have
\be
\label{III.38-b}
\check{u}_k:=\frac{u_k}{\sum_{\om}|A^{\om,k}_{-1}|}=\frac{\ti{u}_k}{\sum_{\om}|A^{\om,k}_{-1}|}+ \frac{A^{\om,k}_{-1}}{\sum_{\om}|A^{\om,k}_{-1}|}\ g^{-1}\ e^{-\sqrt{1-t_k}\,|G|}\ .
\ee
We deduce that on each component $\om$ $\check{u}_\infty$ can be decomposed as follows
\be
\label{III.38-c}
\check{u}_\infty=\ti{u}_\infty+ B^\om\, g^{-1}
\ee
Therefore with the help of Lemma~\ref{lma-unique} we get to a contradiction. This conclude the proof of \eqref{eq: u-tilde-bounded-W12}, i.e. there holds
\be
\label{III.38-d}
\liminf_{k\rightarrow +\infty}\sum_\om\|\nabla \ti{u}_k\|^2_{L^2(\om)}<+\infty\ .
\ee
For any $\omega$, let $\Delta$ be a positive measure connected component of $\p\om\cap \p D^2$. We have
\be
\label{III.39}
\begin{array}{rl}
\ds\limsup_{k\rightarrow+\infty}\lf\|u_k-\ti{u}_k-\dashint_\Delta (u_k-\ti{u}_k)\rg\|_{L^2(\Delta)}&\ds\le C\|\psi\|_{L^2(\p D^2)} +\limsup_{k\rightarrow+\infty}\lf\|\ti{u}_k-\dashint_\Delta \ti{u}_k\rg\|_{L^2(\Delta)}\\[5mm]
\ds\quad\quad\quad&\ds\le C\|\psi\|_{L^2(\p D^2)} + C\ \limsup_{k\rightarrow +\infty}\|\nabla\ti{u}_k\|_{L^2(\om)}\ .
\end{array}
\ee
(here $\tilde{u}_k$ denotes the trace of $\tilde{u}_k\vert_{\omega}$). This gives
\be
\label{III.40}
\begin{array}{rl}
\ds\limsup_{k\rightarrow+\infty}|A^{\om,k}_{-1}|\ \lf\|g^{-1}-\dashint g^{-1}\rg\|_{L^2(\Delta)}<+\infty\ .
\end{array}
\ee
Hence $ \lf\|g^{-1}-\dashint g^{-1}\rg\|_{L^2(\Delta)}\ne 0$ and we deduce
\be
\label{III.41}
\ds\limsup_{k\rightarrow+\infty}|A^{\om,k}_{-1}|<+\infty
\ee
This gives that ${u}_k$ is uniformly bounded in $W^{1,(2,\infty)}(D^2)$. Since all $u_k$ have trace $\psi$ on $\partial D^2$ we can therefore extract a subsequence of $u_k$ converging weakly-$\ast$ in $W^{1,(2,\infty)}(D^2)$ toward a function $u\in \mathfrak{H}$ which satisfies \eqref{II.4}. The fact that $u$ satisfies \eqref{II.5-a} follows from Lemma \ref{lm-exp}. This 
and Theorem~\ref{th-II.1} follows.\\
Finally we discuss how to remove the assumption that for any $\omega$ there holds $\lvert \partial \omega\cap\partial D^2\rvert>0.$ Let $\omega_0$ be one of the connected components of $G\neq 0$ such that $\lvert\partial \omega_0\cap\partial D^2\rvert>0$. Number all other components $\omega$ inductively, in such a way that any $\omega_i$ satisfies $\lvert \partial \omega_i\cap\partial \omega_j\rvert>0$ for some $j<i$ or $\lvert\partial \omega_i\cap\partial D^2\rvert>0$.
For $\omega_0$ we can perform computation \eqref{III.39} to obtain (up to a subsequence)
\begin{align}
\label{eq: control-on-omega0}
    \lim_{k\rightarrow+\infty}|A^{\om_0,k}_{-1}|<+\infty\quad\text{and}\quad \lim_{k\to\infty}\lVert \tilde{u}_k\rVert_{W^{1,2}(\omega_0)}<\infty.
\end{align}
Since all $u_k$ have trace $\psi$ on $\partial D^2$, we can extract a subsequence converging weakly-$\ast$ in $W^{1,2}(\omega_0)$.
Moreover it follows from \eqref{eq: control-on-omega0} that the traces of $u_k$ on $\partial \omega_0$ are uniformly bounded in $H^\frac{1}{2}(\partial \omega_0)$. Now let $\Delta_1$ be a connected component of $\partial \omega_0\cap\partial \omega_1$ of positive measure. We can repeat computation \eqref{III.39} for $\omega_1$ and $\Delta_1$, where on the right hand side instead of $\lVert \psi\rVert_{L^2(\partial D^2)}$ we use $\sup_{k\in \mathbb{N}}\lVert u_k\rVert_{L^2(\partial\omega_0 \cup \partial D^2)}$. This way we obtain that there is a subsequence converging weakly-$\ast$ in $W^{1,2}(\omega_1)$ (as the traces are bounded in $L^2(\partial D^2)$) and the analogous to \eqref{eq: control-on-omega0} holds for $\omega_1$. We can therefore iterate this argument to all the components $\omega$ and then conclude as in the previous case.
\hfill $\Box$
\subsection{A Question}
\begin{Que}
Under the assumption of Lemma~\ref{th-II.1} on the space
\[
{\mathfrak D}:=\lf\{\psi\in C^1(\p D^2,{\C})\ ;\ \int_{\p D^2} \psi\, dg=0   \rg\}
\] 
we consider the map
\[
\begin{array}{rcl}
\ds{\mathfrak L}:\quad{\mathfrak D}&\ds\quad\longrightarrow\quad&\ds {\C}^N\\[3mm]
\psi&\ds\quad\longrightarrow\quad&\ds \lf(\int_{\p \om_j} u\ dg\rg)_{j=1\cdots N}
\end{array}
\]
where $u$ is the solution given by  lemma~\ref{th-II.1}. Compute the dimension of the ${\C}$-Rank of $\mathfrak L$ in generic configurations. Is it true that generically 
\be
\label{III.42}
\mbox{Rank}_{\C}({\mathfrak L})=N-1?
\ee
\end{Que}
\section{The case of 2 singular points of opposite degrees}

We consider the case of two point singularities. That is, we assume that $D^2=\om_+\cup \om_-\cup\{G=0\}$ where $\om_+$ and $\om_-$ are diffeomorphic to $D^2$ and correspond respectively to the domains
where $G>0$ and $G<0$. Moreover there exists $p_\pm\in \om_\pm$ such that
 \be
\label{r-II.42}
\lf\{
\begin{array}{l}
\ds \Delta G= 2\pi \, \mp \delta_{p_\pm}\quad\mbox{ in }\om_\pm\\[5mm]
\ds G=0\quad\mbox{ on }\ \p \om_\pm\ .
\end{array}
\rg.
\ee
Let $\psi\in C^1(\p D^2,{\C})$ such that
\be
\label{r-II.42-a}
\int_{\p D^2}\psi\ dg=0\ .
\ee
Let $u$ be the $W^{1,(2,\infty)}$ solution of
\[
\lf\{
\begin{array}{l}
\ds\Delta u+i\,\nabla^\perp G\cdot \nabla u=0\quad\mbox{ in }D^2\\[5mm]
\ds u=\psi\quad\mbox{ on }\p D^2
 \end{array}
 \rg.
\]
Because of lemma~\ref{lm-exp} there exist $(A^+_j)_{j\in{\Z}}$ and $(A^-_j)_{j\in{\Z}}$
\be
\label{r-II.43}
\begin{array}{l}
\ds u(x)=\lf\{
\begin{array}{l}
\ds\sum_{j\in \Z} A^+_j \ e^{\sqrt{j^2+\,j}\, G}\, g^j\quad\mbox{ on }\om_+\\[5mm]
\ds\sum_{j\in \Z} A^-_j \ e^{-\,\sqrt{j^2+\,j}\, G}\, g^j\quad\mbox{ on }\om_-
\end{array}\rg.
\end{array}
\ee
where respectively
\be
\label{r-II.44}
A^+_j:=(2\pi\,i)^{-1}\,\int_{\p\om_+} u(x) \,g^{-j-1}\, dg\  \quad\mbox{ and }\quad A^-_j:=-(2\pi\,i)^{-1}\ 
\int_{\p\om_-} u(x) \, g^{-j-1}\, dg\ \ .
\ee
Hence this gives in particular
\be
\label{r-II.45}
A^+_j-A^-_j:=(2\pi\,i)^{-1}\,\int_{\p D^2} u(x) \,g^{-j-1}\, dg\ .
\ee
In particular, because of oour assumptions we have
\be
\label{r-II.46}
A^+_{-1}-A^-_{-1}=0\ .
\ee
Let $A_\psi:=A^+_{-1}=A^-_{-1}$. We have that
\be
\label{r-II.47}
\ti{u}:=u- A_\psi\, g^{-1}
\ee
is a $W^{1,2}(D^2)$ solution to
\[
\lf\{
\begin{array}{l}
\ds\Delta \ti{u}+i\,\nabla^\perp G\cdot \nabla \ti{u}=0\quad\mbox{ in }D^2\\[5mm]
\ds \ti{u}=\psi-A_\psi\,g^{-1}\quad\mbox{ on }\p D^2
 \end{array}
 \rg.
\]
We deduce in particular that for any $\la\in {\C}$
\[
{\mathfrak L}(\psi-A_\psi\,g^{-1}+\la\psi)=(2\pi\la,2\pi\la)
\]
This implies that the rank of ${\mathfrak L}$ on $\mathfrak D$ is 1.

\appendix
\addcontentsline{toc}{section}{Appendices}
\section*{Appendix}

\section{Optimal Wente Type Inequality}
\reset
\begin{Lma}
\label{lma-1}
Let $(d_l)_{l=1\cdots N}$ such that $d_i=\pm1$ and $\sum_{l=1}^N d_i=0$. Let $p_l$ be $N$ distinct points in $D^2$ and let $G$ be the Green's function solution to 
\be
\label{A-1}
\lf\{
\begin{array}{rl}
\ds\Delta G&\ds=2\pi\ \sum_{l=1}^N\ d_l\, \delta_{p_l}\quad\mbox{in }{\mathcal D}'(D^2)\\[5mm]
\ds G&\ds=0\quad\mbox{ on }\p D^2\ .
\end{array}
\rg.
\ee
Assume that for ${\mathcal H}^1-$almost every $t\in{\R}$ each of the connected components $\om_j(t)$ of $\{ G(x)\le t\}$  satisfy
\be
\label{A-2}
\int_{\p \om_j(t)}\frac{\p G}{\p \nu}=2\pi\ ,
\ee
then for any $a$ and $b$ in $W^{1,2}(D^2,{\R})$ there holds
\be
\label{A-3}
\lf|\sum_{i=1}^N\,d_i\,\varphi(p_i)\rg|\le \frac{1}{4\pi}\int_{D^2}|\nabla a|^2+|\nabla b|^2\ dx^2
\ee
where
\be
\label{A-4}
\lf\{
\begin{array}{rl}
\ds\Delta \varphi&\ds=\nabla a\cdot\nabla^\perp b= \p_{x_2}a\,\p_{x_1} b-\p_{x_1}a\,\p_{x_2}b\quad\mbox{in }{\mathcal D}'(D^2)\\[5mm]
\ds \varphi&\ds=0\quad\mbox{ on }\quad\p D^2.
\end{array}
\rg.
\ee
Moreover equality holds in (\ref{A-3}) only if $(a,b)$ are constants.
\end{Lma}
\noindent{\bf Proof of lemma~\ref{lma-1}.} We follow closely \cite{Top}. We first assume that $a$ and $b$ are smooth. Let $v:=(a,b)$. This gives in particular
\be
\label{A-5}
\mbox{det}(\nabla v)=\p_{x_1}a\,\p_{x_2}b-\p_{x_2}a\,\p_{x_1} b\ .
\ee
Hence, using the co-area formula we have
\be
\label{A-6}
\begin{array}{l}
\ds-\,2\pi\, \sum_{i=1}^N\,d_i\,\varphi(p_i)=-\,\int_{D^2} G\, \Delta\varphi\ dx^2=\int_{D^2} G\, \mbox{det}(\nabla v)\ dx^2\\[5mm]
\ds=\int_{-\infty}^{+\infty}t\ dt\,\int\, \frac{\mbox{det}(\nabla v)}{|\nabla G|}\ d{\mathcal H}^1\res G^{-1}(\{t\})\ .
\end{array}
\ee
Observe that
\be
\label{A-7}
\begin{array}{l}
\ds\int\, \frac{\mbox{det}(\nabla v)}{|\nabla G|}\ d{\mathcal H}^1\res G^{-1}(\{t\})=\frac{d}{dt}\int_{\{G(x)\le t\}}\mbox{det}(\nabla v)\ dx^2\\[5mm]
\ds=\frac{d}{dt}\int_{\{G(x)= t\}} a\,\p_\tau b\, dl
\end{array}
\ee
where $\tau$ is the oriented unit tangent along $\{G(x)= t\}$ and $dl$ the associated length element. The integration by part in (\ref{A-6}) (easily justified - see \cite{Top}) is giving
\be
\label{A-8}
\begin{array}{l}
\ds 2\pi\, \sum_{i=1}^N\,d_i\,\varphi(p_i)=\int_{-\infty}^{+\infty}\ dt \int_{\{G(x)\le t\}} \mbox{det}(\nabla v)\ dx^2
\end{array}
\ee
We can assume that $t$ is a regular value for $G$ and we decompose the upper level sets $\{ G(x)\le t\}$ into it's connected components $(\om_j(t))_{j\in J}$. The isoperimetric inequality
(see \cite{Top}) gives
\be
\label{A-9}
\lf|\int_{\{G(x)\le t\}} \mbox{det}(\nabla v)\ dx^2\rg|\le\sum_{j\in J} \lf|\int_{\om_j(t)}\  \mbox{det}(\nabla v)\ dx^2\rg|\le\frac{1}{4\pi}\, \sum_{j\in J}\ \lf|\int_{\p \om_j(t)}|\p_\tau v|\ dl\rg|^2
\ee
Using Cauchy Schwartz we have
\be
\label{A-10}
\lf|\int_{\p \om_j(t)}|\p_\tau v|\ dl\rg|^2\le \int_{\p \om_j(t)}|\nabla G|\ dl\ \int_{\p \om_j(t)}\frac{|\p_\tau v|^2}{|\nabla G|}\ dl\
\ee
By assumption 
\be
\label{A-11}
\int_{\p \om_j(t)}|\nabla G|\ dl=\int_{\p \om_j(t)}\frac{\p G}{\p\nu}\ dl=2\pi\ .
\ee
Hence we deduce
\be
\label{A-12}
\lf|\int_{\{G(x)\le t\}} \mbox{det}(\nabla v)\ dx^2\rg|\le\frac{1}{2}\int_{\p \om_j(t)}\frac{|\p_\tau v|^2}{|\nabla G|}\ dl
\ee
Integrating over $t$ and combining with (\ref{A-8}) gives (\ref{A-3}). Equality holds if $\om_j$ are discs hence there is exactly one singularity $p_j=0$ and if $\p_\theta v$ is constant as well as  $\p_r v=0$. This implies $\nabla v=0$.\hfill $\Box$

\section{A Uniqueness Result for Critical Points of $L_1$}
\reset

\begin{Lma}
\label{lma-unique}
Let $(d_l)_{l=1\cdots N}$ such that $d_l=\pm1$ and $\sum_{l=1}^N d_l=0$. Let $p_i$ be $N$ distinct points in $D^2$ and let $G$ be the Green's function solution to 
\be
\label{B.1}
\lf\{
\begin{array}{rl}
\ds\Delta G&\ds=2\pi\ \sum_{l=1}^N\ d_l\, \delta_{p_l}\quad\mbox{in }{\mathcal D}'(D^2)\\[5mm]
\ds G&\ds=0\quad\mbox{ on }\quad\p D^2\ .
\end{array}
\rg.
\ee
Assume that for ${\mathcal H}^1-$almost every $t\in{\R}$ each of the connected components $\om_i(t)$ of $\{ G(x)\le t\}$  satisfy
\be
\label{B.2}
\int_{\p \om_j(t)}\frac{\p G}{\p \nu}=2\pi\ ,
\ee
and $\nabla G\ne 0$ away from $G^{-1}(\{0\})$ in $D^2\setminus \{p_1\cdots p_N\}$. Let $u\in W^{1,(2,\infty)}_0(D^2,{\C})$ be a solution of 
\be
\label{B.3}
\lf\{
\begin{array}{l}
\ds-\Delta u\ds= i\, \nabla^\perp G\cdot \nabla u\quad\mbox{ in }{\mathcal D}'(D^2)\\[5mm]
\ds u\ds=0\quad\mbox{ on }\p D^2\ .
\end{array}
\rg.
\ee
Assume that on each connected component $\om$ of $G\ne 0$ there exists $\ti{u}\in W^{1,2}(\om,{\C})$ and $A^\om\in {\C}$ such that
\[
u=\ti{u}+ A^\om\,g^{-1} \quad\mbox{ in }{\mathcal D}'(\om)
\]
Then $u \equiv 0$ on $D^2$.\hfill $\Box$
\end{Lma}
\noindent{\bf Proof of lemma~\ref{lma-unique}} Observe that on any $\om_i$   $\ti{u}$ satisfies
\be
\label{B.4}
-\Delta \ti{u}= i\, \nabla^\perp G\cdot \nabla \ti{u}\quad\mbox{ in }{\mathcal D}'(\om_i)
\ee
Multiplying by $u$ itself and integrating by parts\footnote{This multiplication is justified by the fact that $\ti{u}\in C^{1,\sqrt{2}-1}(\om_i)$ thanks to lemma~\ref{lm-regularity}.}, using the fact that $G=0$ on $\p\om_i$ and that $u=0$ on $\p\om_i\cap\p D^2$, is giving
\be
\label{B.5}
\int_{\om_i}\nabla u\cdot\nabla\ti{u}\ dx^2-\sum_{j}\int_{\p \om_i\cap\p \om_j} u\cdot\frac{\p \ti{u}}{\p\nu_i}\ dl=-\,\int_{\om_i} G \lf<i\,\nabla^\perp u\cdot\nabla \ti{u}\rg>\ dx^2
\ee
where $\nu_i$ is the unit normal to $\p\om_i$ pointing outside $\om_i$. Observe that on each $\om_i$, after applying the Riemann Mapping and Fourier decompositions there holds
\be
\label{B.6}
\int_{\om_i}\nabla u\cdot\nabla\ti{u}\ dx^2=\int_{\om_i}\nabla \ti{u}\cdot\nabla\ti{u}\ dx^2\quad\mbox{ and }\quad \int_{\om_i} G \lf<i\,\nabla^\perp u\cdot\nabla \ti{u}\rg>\ dx^2=\int_{\om_i} G \lf<i\,\nabla^\perp \ti{u}\cdot\nabla \ti{u}\rg>\ dx^2
\ee
Moreover, since $\p_{\nu_i}g=0$ on $\p\om_i$ there holds
\be
\label{B.7}
\sum_{j}\int_{\p \om_i\cap\p \om_j} u\cdot\frac{\p \ti{u}}{\p\nu_i}\ dl=\sum_{j}\int_{\p \om_i\cap\p \om_j} u\cdot\frac{\p {u}}{\p\nu_i}\ dl\ .
\ee
Using the fact that
\be
\label{B.8}
\sum_{j}\int_{\p \om_i\cap\p \om_j} u\cdot\frac{\p {u}}{\p\nu_i}\ dl=\int_{\om_i}\operatorname{div}\lf(u\cdot\nabla u\rg)\ dx^2\ ,
\ee
by summing over $i$ we obtain
\be
\label{B.8-a}
\sum_i\sum_{j}\int_{\p \om_i\cap\p \om_j} u\cdot\frac{\p {u}}{\p\nu_i}\ dl=\sum_i\int_{\om_i}\operatorname{div}\lf(u\cdot\nabla u\rg)\ dx^2=\int_{D^2}\operatorname{div}\lf(u\cdot\nabla u\rg)\ dx^2=0\ .
\ee
Using the identity (\ref{B.6}) we finally get
\be
\label{B.9}
\sum_{i}\int_{\om_i}\nabla \ti{u}\cdot\nabla\ti{u}\ dx^2+\int_{\om_i} G \lf<i\,\nabla^\perp \ti{u}\cdot\nabla \ti{u}\rg>\ dx^2=0\ .
\ee 
Thanks to the optimal Wente inequality we have for each $i$
\be
\label{B.10}
\int_{\om_i}\nabla \ti{u}\cdot\nabla\ti{u}\ dx^2+\int_{\om_i} G \lf<i\,\nabla^\perp \ti{u}\cdot\nabla \ti{u}\rg>\ dx^2\ge 0\ .
\ee
Hence we deduce that for any $i$ there holds
\be
\label{B.11}
\int_{\om_i}\nabla \ti{u}\cdot\nabla\ti{u}\ dx^2+\int_{\om_i} G \lf<i\,\nabla^\perp \ti{u}\cdot\nabla \ti{u}\rg>\ dx^2= 0\ .
\ee
The case of equality in the Wente inequality implies $\ti{u}=0$ and $u$ is a multiple of $g$. The map $g$ cannot be identically equal to zero on $\p D^2$ hence we have obtained that $u=0$ and the lemma is proved.\hfill $\Box$

\section{$L^\infty$ estimates}
In this appendix we show that under the assumptions of Lemma \ref{lm:main-lemma-disk} (with an additional control of the trace in $L^\infty)$ also the $L^\infty$ norm of the sequence tends to zero (up to subsequences). This result is not needed for the proof of our main theorems, but might be useful for future developments.
\begin{Lma}
\label{lem: vanishing-W12-norm}
    Let $\{v_k\}_{k\in \mathbb{N}}$ be a sequence of functions in $W^{1,2}(D^2)$, $\{t_k\}_{k\in \mathbb{N}}$ a sequence in $(0,1)$ such that $\lim_{k\to\infty}t_k=1$.
    Assume that
    \begin{enumerate}
        \item $-\Delta v_k=it_k\nabla^\perp G\nabla v_k$ in $D^2$ $\forall k\in \mathbb{N}$,
        \item $\lVert \nabla v_k\rVert_{L^2(D^2)}<C$ for any $k\in \mathbb{N}$.
        \item $v_k\vert_{\partial D^2}\to 0$ in $H^\frac{1}{2}\cap L^\infty(\partial D^2)$.
    \end{enumerate}
    Then, up to subsequences,
    \begin{align}
        \lim_{k\to\infty}\lVert  v_k\rVert_{L^\infty(D^2)}=0.
    \end{align}
\end{Lma}
\begin{proof}
We extract a subsequence  still denoted $v_k$ such that $v_k$ weakly converges in $W^{1,2}(D^2,{\C})$ to a limit $v_\infty$.

\medskip

\noindent{\bf Case 1 : $\mathbf v_\infty\ne 0$.}  In this case we have by linearity that $v_\infty\in W^{1,2}(D^2,{\C})$ solves

\be
\lf\{
\begin{array}{l}
\ds-\Delta v_\infty\ds= i\, \nabla^\perp G\cdot \nabla v_\infty\quad\mbox{ in }D^2\\[5mm]
\ds v_\infty\ds=0\quad\mbox{ on }\p D^2\ .
\end{array}
\rg.
\ee
Multiplying by $v_\infty$ and integrating by parts gives
\be
\label{int-parts}
 L_1(v_\infty):=\frac{1}{2}\int_{D^2}|\nabla v_\infty|^2\ dx^2+\frac{1}{2}\int_{D^2}G\, \lf<i\,\nabla v_\infty,\nabla^\perp v_\infty\rg>\ dx^2=0
\ee
Let $c_\infty$ be the solution to
\be
\lf\{
\begin{array}{l}
\ds\Delta c_\infty\ds=\lf<i\,\nabla v_\infty,\nabla^\perp v_\infty\rg>\quad\mbox{ in }D^2\\[5mm]
\ds c_\infty=0\quad\mbox{ on }\p D^2\ .
\end{array}
\rg.
\ee
Inserting this identity in \eqref{int-parts} gives
\be
\int_{D^2}|\nabla v_\infty|^2\ dx^2=2\pi\, c_\infty(0)\ .
\ee
This is a case of equality for the $L^\infty$ Wente estimate and thanks to \cite{Top} (see the observation after theorem 1) this implies that $v_\infty=0$ which is a contradiction.\\
\noindent{\bf Case 2 : $\mathbf v_\infty= 0$.} Since for any $k\in \mathbb{N}$ $v_k$ solves \eqref{I.7} with a trace converging to zero in $H^\frac{1}{2}\cap L^\infty(\partial D^2)$ we have, up to a subsequence,
\be
\label{convergence-to-zero}
v_k\rightarrow 0\quad\mbox{ in }(W^{1,2}_{loc}\cap L^\infty_{loc})(\ov{D^2}\setminus \{0\})\cap\bigcap_{l\in \mathbb{N}}C^l_{loc}({D^2}\setminus \{0\})\ .
\ee
We shall divide this case in two sub-cases.\\
\noindent{\bf Case 2.1 : $\mathbf v_\infty= 0$ and
\be
\label{I.22}
\limsup_{k\rightarrow +\infty}\sup_{r>0}\int_{B^2_r(0)\setminus B^2_{r/2}(0)}|\nabla v_k|^2\ dx^2>0\ .
\ee}
We call this case the ``semi-vanishing case''. Hence there exists $\delta>0$ a subsequence that we keep denoting $v_k$ and $r_k\rightarrow 0$ such that
\be
\label{I.23}
\liminf_{k\rightarrow +\infty}\int_{B^2_{r_k}(0)\setminus B_{r_k/2}(0)}|\nabla v_k|^2\ dx^2\ge\delta
\ee
Let $\hat{v}_k(x):=v_k(r_k\, x)$. We have for any $R>0$
\be
\label{I.24}
\int_{B_R(0)}|\nabla \hat{v}_k|^2\ dx^2\le 1
\ee
The sequence $\hat{v}_k$ satisfies for any $R>0$ and $k$ large enough
\be
\label{I.25}
-\Delta \hat{v}_k\ds= i\, t_k\,\nabla^\perp {G}\cdot \nabla \hat{v}_k\quad\mbox{ in }B_R(0)
\ee
By standard elliptic bootstrap argument $\hat{v}_k$
is pre-compact in $C^l_{loc}({\C}\setminus\{0\})$ for any $l\in {\N}$. Because of (\ref{I.23}) we have
\be
\label{I.27}
\int_{B_1(0)\setminus B_{1/2}(0)}|\nabla \hat{v}_k|^2\ dx^2\ge\delta>0\ .
\ee
Hence we can extract a subsequence, still denoted $k$, such that $\hat{v}_k$ converges strongly towards $\hat{v}_\infty\ne 0$ in $C^l_{loc}({\C}\setminus\{0\})$ for any $l\in {\N}$ and weakly in $W^{1,2}({\C},{\C})$.
The map $\hat{v}_\infty$ satisfies
\be
\label{I.28}
\begin{array}{l}
\ds-\Delta \hat{v}_\infty\ds= i\, \nabla^\perp {G}\cdot \nabla \hat{v}_\infty\quad\mbox{ in }{\C}\ .
\end{array}
\ee
Multiplying by $\hat{v}_\infty$ and integrating by parts gives
\be
\label{I.29}
 L_1(\hat{v}_\infty):=\frac{1}{2}\int_{{\C}}|\nabla \hat{v}_\infty|^2\ dx^2+\frac{1}{2}\int_{\C}{G}\, \lf<i\,\nabla \hat{v}_\infty,\nabla^\perp \hat{v}_\infty\rg>\ dx^2=0\ .
\ee
Let $\hat{c}_\infty$ be the $W^{1,2}$ solution to
\be
\label{I.30}
\begin{array}{l}
\ds\Delta \hat{c}_\infty\ds=\lf<i\,\nabla \hat{v}_\infty,\nabla^\perp \hat{v}_\infty\rg>\quad\mbox{ in }{\C}\ ,
\end{array}
\ee
given by
\[
\hat{c}_\infty:=\frac{1}{2\pi}\,\log|x|\star \lf<i\,\nabla \hat{v}_\infty,\nabla^\perp \hat{v}_\infty\rg>\ .
\]
Inserting (\ref{I.30})  in (\ref{I.29}) gives
\be
\label{I.31}
\int_{D^2}|\nabla \hat{v}_\infty|^2\ dx^2=2\pi\, \hat{c}_\infty(0)\ .
\ee
This is a case of equality for the $L^\infty$ Wente estimate and thanks to \cite{Top}  this implies that $\hat{v}_\infty=0$ which is a contradiction.

\medskip 

We are left with the ``totally vanishing case'' which is the most delicate one for ``tracking'' the limiting ``distribution'' of the $L^2$ energy of $\nabla v_k$:\\
\noindent{\bf Case 2.1 : $\mathbf v_\infty= 0$ and
\be
\label{I.32}
\limsup_{k\rightarrow +\infty}\sup_{r>0}\int_{B^2_r(0)\setminus B^2_{r/2}(0)}|\nabla v_k|^2\ dx^2=0\ .
\ee}
In this case, Lemma \ref{lem: vanishing-Linfty-norm} implies that $v_k$ tends to zero uniformly in a ball around the origin (up to a subsequence. Together with \eqref{convergence-to-zero} this proves the Lemma.
\end{proof}

\begin{Lma}
\label{lem: vanishing-Linfty-norm}
    Let $\{v_k\}_{k\in \mathbb{N}}$ be a sequence of functions in $W^{1,2}(D^2)$, $\{t_k\}_{k\in \mathbb{N}}$ a sequence in $(0,1)$ such that $\lim_{k\to\infty}t_k=1$. Let $\rho\in \left(0,\frac{1}{2}\right)$.
    Assume that
    \begin{enumerate}
        \item $-\Delta v_k=it_k\nabla^\perp G\nabla v_k$ in $D^2$ $\forall k\in \mathbb{N}$,
        \item $\lim_{k\to\infty}\lVert v_k\rVert_{C^2(\partial B_\rho(0))}=0$.
        \item $\displaystyle \limsup_{k\rightarrow +\infty}\sup_{0<r<1}\int_{B^2_r(0)\setminus B^2_{r/2}(0)}|\nabla v_k|^2\ dx^2=0$.
    \end{enumerate}
    Then
    \begin{align}
        \lim_{k\to\infty}\lVert v_k\rVert_{L^\infty(B_\rho(0))}=0.
    \end{align}
\end{Lma}

\begin{proof}
    Let
\[
\ov{v}_k(r):=\dashint_0^{2\pi} v_k(r,\theta)\ d\theta\ .
\]
Observe that $v_k$ satisfies
\be
\label{I.34}
\begin{array}{rl}
\ds\frac{\p^2 v_k}{\p r^2}+\frac{1}{r}\frac{\p v_k}{\p r}+\frac{1}{r^2}\frac{\p^2v_k}{\p\theta^2}&\ds=t_k\, \frac{i}{r}\ \frac{\p G}{\p r}\,\frac{\p v_k}{\p \theta}-t_k\, \frac{i}{r}\ \frac{\p v_k}{\p r}\,\frac{\p G}{\p \theta}\\[5mm]
\ds &\ds=t_k\, \frac{i}{r}\ \frac{\p}{\p\theta}\lf( v_k\ \frac{\p G}{\p r}\rg)-t_k\, \frac{i}{r}\ \frac{\p}{\p r}\lf( v_k\ \frac{\p G}{\p \theta}\rg)\\[5mm]
\ds &\ds=t_k\, \frac{i}{r}\ \frac{\p}{\p\theta}\lf( v_k\ \frac{\p G}{\p r}\rg)
\end{array}
\ee
Taking the average with respect to $\theta$ gives
\be
\label{I.35}
\begin{array}{rl}
\ds\Delta\ov{v}_k=\frac{\p^2 \ov{v}_k}{\p r^2}+\frac{1}{r}\frac{\p \ov{v}_k}{\p r}=0\end{array}
\ee
Since $\ov{v}_k$ converges to $0$ in $C^2(\partial B_\rho(0))$, we deduce that
\be
\label{I.39}
\limsup_{k\rightarrow +\infty}\int_{B_{\rho}(0)}|\nabla^2 \ov{v}_k|^2+|\nabla \ov{v}_k|^2+|\ov{v}_k|^2\ dx^2=0
\ee
By Sobolev embedding we have in particular that
\be
\label{I.40}
\lim_{k\rightarrow +\infty}\|\ov{v}_k\|_{L^\infty(B_{\rho}(0))}=0\ .
\ee
Now let $k\in \mathbb{N}$. Since $v_k=\overline{v}_k$ has zero average on every annulus centered in zero, by Poincar\'e-Sobolev inequality\footnote{Here we are also making use of the following estimate: let $\varphi$ be a cut-off function supported in $B_2$ and equal to one in $B_1$. Then
\begin{align}
    \int_{B_1}\lvert \operatorname{Hess} v\rvert^2\leq C\int_{B_2}\lvert D(\varphi\nabla v)\rvert^2\leq C\int_{B_2}\lvert d(\varphi d v)\rvert^2+\lvert d^\ast (\varphi d v)\rvert^2\leq C\int_{B_2}\lvert \Delta v\rvert^2+\lvert\nabla v\rvert^2.
\end{align}
} we have for any $r<\rho/2$
\begin{align}
\label{I.41}
\|v_k-\ov{v}_k\|_{L^\infty(B_r(0)\setminus B_{r/2}(0))}= &
\operatorname{osc}_{B_r(0)\setminus B_{r/2}(0))}(v_k-\ov{v}_k)\\ \leq & C\left(r\lVert \operatorname{Hess} v_k\rVert_{L^2(B_r(0)\setminus B_{r/2}(0))}+\lVert \nabla v_k\rVert_{L^2(B_r(0)\setminus B_{r/2}(0))}\right)\\
\nonumber
\leq & C\left(r\lVert\Delta v_k\rVert_{L^2(B_{2r}(0)\setminus B_{r/4}(0)))}+\lVert \nabla v_k\rVert_{L^2(B_r(0)\setminus B_{r/2}(0))}\right)
\end{align}
Hence using the fact that $|\nabla G|\le C\, r^{-1}$ and the equation satisfied by $v_k$ we obtain
\be
\label{I.42}
\sup_{0<r<\rho/2}\|v_k-\ov{v}_k\|^2_{L^\infty(B_r(0)\setminus B_{r/2}(0))}\le C\ \sup_{0<r<\rho/2}\int_{B_{2r}(0)\setminus B_{r/4}(0))}|\nabla v_k|^2\ dx^2\ .
\ee
Using  assumption 3 we deduce that
\be
\label{I.43}
\lim_{k\rightarrow +\infty}\sup_{0<r<\rho/2}\|v_k-\ov{v}_k\|^2_{L^\infty(B_r(0)\setminus B_{r/2}(0))}=0
\ee
Combining (\ref{I.40}) and (\ref{I.43}) gives
\be
\label{I.44}
\lim_{k\rightarrow+\infty}\|v_k\|_{L^\infty(B_\rho)}=0\ .
\ee
This concludes the proof of Lemma~\ref{lem: vanishing-Linfty-norm}.
\end{proof}

\end{document}